\documentclass[11pt]{article}

\usepackage[T1]{fontenc}
\usepackage{amsfonts}
\usepackage{amsmath}
\usepackage{amssymb}
\usepackage{amsthm}
\usepackage{bbm}
\usepackage{bm}
\usepackage{mathrsfs}
\usepackage{verbatim}
\usepackage{setspace}
\usepackage{color}
\usepackage{pdfsync}
\usepackage{enumitem}
\usepackage{graphicx}
\usepackage{stmaryrd}
\usepackage{float} %
\usepackage{tikz}
\usetikzlibrary{automata, positioning}
\usepackage{tcolorbox}

\theoremstyle{plain}
\newtheorem{theorem}{Theorem}[section]

\newtheorem{lemma}[theorem]{Lemma}
\newtheorem{corollary}[theorem]{Corollary}

\theoremstyle{definition}
\newtheorem{definition}[theorem]{Definition}
\newtheorem{remark}[theorem]{Remark}

\newtheorem{example}[theorem]{Example}

\theoremstyle{remark}

\renewenvironment{thebibliography}[1]{%
\begin{oldthebibliography}{#1}%
\setlength{\baselineskip}{.9em}
\linespread{1}
\small
\setlength{\parskip}{0.3ex}%
\setlength{\itemsep}{.2em}%
}%
{%
\end{oldthebibliography}%
}

\newcommand{\N}{\mathbb{N}}
\renewcommand{\P}{P}

\newcommand{\R}{\mathbb{R}}

\newcommand{\X}{\mathbb{X}}
\newcommand{\T}{\mathbb{T}}

\newcommand{\cA}{\mathcal{A}}

\newcommand{\cF}{\mathcal{F}}

\newcommand{\cL}{\mathcal{L}}

\DeclareMathOperator*{\esssup}{ess\, sup}

\newcommand{\as}{\mbox{a.s.}}

\numberwithin{equation}{section}

\usepackage[pdfborder={0 0 0}]{hyperref}
\hypersetup{
  urlcolor = black,
  pdfauthor = {Marcel Nutz, Yuchong Zhang},
  pdfkeywords = {Conditional Optimal Stopping; Time-inconsistency; Equilibrium},
  pdftitle = {Conditional Optimal Stopping: A Time-Inconsistent Optimization},
  pdfsubject = {Conditional Optimal Stopping: A Time-Inconsistent Optimization},
  pdfpagemode = UseNone
}

\begin{document}

\title{\vspace{-2.5em}
  Conditional Optimal Stopping:\\ A Time-Inconsistent Optimization
\date{\today}
\author{
  Marcel Nutz%
  \thanks{
  Depts.\ of Statistics and Mathematics, Columbia University, mnutz@columbia.edu. Research supported by an Alfred P.\ Sloan Fellowship and NSF Grant DMS-1812661. MN is grateful to Pierre-Louis Lions and Abdoulaye Ndiaye for helpful discussions.
  }
  \and
  Yuchong Zhang%
  \thanks{Dept.\ of Statistical Sciences, University of Toronto, yuchong.zhang@utoronto.ca. %
  }
 }
}
\maketitle \vspace{-.8em}

\begin{abstract}
Inspired by recent work of P.-L.~Lions on conditional optimal control, we introduce a problem of optimal stopping under bounded rationality: the objective is the expected payoff at the time of stopping, conditioned on another event. For instance, an agent may care only about states where she is still alive at the time of stopping, or a company may condition on not being bankrupt. We observe that conditional optimization is time-inconsistent due to the dynamic change of the conditioning probability and develop an equilibrium approach in the spirit of R.~H.~Strotz' work for sophisticated agents in discrete time. Equilibria are found to be essentially unique in the case of a finite time horizon whereas an infinite horizon gives rise to non-uniqueness and other interesting phenomena. We also introduce a theory which generalizes the classical Snell envelope approach for optimal stopping by considering a pair of processes with Snell-type properties.
\end{abstract}

\vspace{0.9em}

{\small


}

\section{Introduction}\label{se:intro}

The classical optimal stopping problem is to maximize the expected payoff $E[G_{\tau}]$ over all stopping times $\tau$, where $G=(G_{t})$ is a given adapted process. In this paper, we propose to study a criterion that conditions on a given stopping time $\sigma$ not being reached at the time~$\tau$: 
\begin{equation}\label{eq:introPrecomm}
  \sup_{\tau}\frac{E[G_\tau 1_{\{\tau\lhd \sigma\}}]}{P(\tau\lhd \sigma)} 
\quad\mbox{where} \quad 
  \tau \lhd \sigma \:\Leftrightarrow\:  \tau < \sigma \mbox{ or } \sigma=\infty.
\end{equation} 
When the model is based on a Markov chain $X$, a natural choice of $\sigma$ is the first exit time from a given set $B$. If, for instance, the stopping decision is made by a company, one application is that $X$ being in~$B$ indicates solvency so that $\sigma$ is the time of bankruptcy. Indeed, the company may only care about states where the stopping payoff happens before~$\sigma$ as the company no longer exists in the other states. Or, for an individual making a financial decision, $\sigma$ may be the time of death, then the model expresses that she only cares about states where the payoff happens while she is alive.

It is typically not possible to model such a conditional problem as a classical optimal stopping problem, except in the trivial case where the conditioning event does not depend on the stopping time~$\tau$. The classical framework would require us to model this as an exit time problem where a specific payoff is assigned to the exit event (that is, a value $G_{t}$ for $t\geq\sigma$). %
E.g., for the individual facing possible death, we are unable to simply say, ``I don't care what happens after I die.'' Instead, we have to assign a specific payoff at death. Even if the modeler were willing to fix some value in order to be ``pragmatic,'' it may be hard to make a justifiable choice and the solution of the optimization will typically depend on it.

This paper is inspired by recent work of P.-L. Lions which introduces the optimal control of conditioned processes~\cite{Lions.17}. There, the main example is controlling the drift of a Brownian motion and the payoff is conditioned on the process staying inside a given domain. The problem is cast as an optimal control problem of Fokker--Planck equations, a particular type of mean field game problem with coupling through the final condition. The limit towards the classical case, where the domain tends to $\R^{d}$, is given particular attention. While it is observed that optimal controls depend on the starting point, the question of time-consistency is not raised.

In the present paper, we introduce optimal stopping with conditioning, a novel problem to the best of our knowledge. One of our first observations is that the problem is time-inconsistent in the sense of Strotz~\cite{Strotz.55}: if an agent determines an optimal strategy at time $t=0$ and reconsiders her decision at a later time taking into account her present state, she may contradict her previous decision and find that her strategy is no longer optimal. In this setting where the dynamic programming principle does not hold, there is more than one notion of optimization. The \emph{precommitted} problem is to optimize the expected payoff at $t=0$, assuming that the decision will not be challenged later on; i.e., the agent ``commits'' to the initial choice. (The theory of~\cite{Lions.17} corresponds to this notion.) In Strotz' terminology, a \emph{sophisticated} agent without a commitment device is aware of the fact that her ``future selves'' may overturn her current plan. Thus, she takes this as a constraint for a ``strategy of consistent planning'': she chooses her behavior ignoring plans that she knows her future selves will not carry out; that is, she  selects an action such that her future incarnations have no incentive to deviate. The resulting time-consistent strategy is called subgame perfect Nash equilibrium, and this is the notion that we will focus on. A different interpretation follows the literature on intergenerational models or overlapping-generations models (see \cite{Samuelson.58} and the work thereafter) where future decisions are taken by subsequent generations rather than other selves. For instance, a government agency may want to take into account future presidential terms and opt for policies which will not be reversed after the next election. 

Beyond being interesting in and of itself, conditional stopping may also help to shed more light on the conditioned control of processes, since optimal stopping is often more tractable than control.
\subsection{Literature}\label{se:lit}

Following the early work of \cite{Strotz.55}, a rich literature involving time-inconsistency has emerged in economics. For instance, \cite{Pollak.68} reconsiders Strotz' concept in a setting with non-exponential discounting when the number of decision points changes, and \cite{PelegYaari.73} studies preferences that change over time. Non-standard discounting (in particular hyperbolic) and time preferences (such as habit formation) are the most frequent reasons for time-inconsistency in this literature; see \cite{FrederickLoewensteinODonoghue.02} for an overview.
The models are mostly formulated in discrete time with finite or infinite time horizon. Time-inconsistency also arises when the optimization objective involves a nonlinear function of an expectation, such as the mean-variance criterion in~\cite{BasakChabakauri.10}, or a probability distortion as in~\cite{Barberis.12,HeZhou.11,JinZhou.08}. (A probability distortion corresponds to an optimization objective that over- or underemphasizes events relative to their objective probability.)

The pioneering work of~\cite{EkelandLazrak.06,EkelandLazrak.10} has initiated the study on how to define and obtain equilibrium strategies for the optimal control of continuous-time processes, using the example of Ramsey's problem when the planner uses non-exponential discounting. In the continuous setting,  varying a control at a single instance in time is meaningless since it does not affect the diffusion. The authors develop a first-order criterion which corresponds to variations of the control over a short time interval, meaning that agents can commit for a short period. This has led to a number of works, including portfolio optimization with non-exponential discounting~\cite{EkelandMbodjiPirvu.12, EkelandPirvu.08}, mean-variance portfolio selection~\cite{BjorkMurgociZhou.14, Czichowsky.13} and general linear--quadratic control~\cite{HuJinZhou.12,HuJinZhou.17}. Nevertheless, this concept of equilibrium is not the only one possible; in particular, first-order conditions are not sufficient for optimality in general. The recent study~\cite{HuangZhou.18} introduces a stronger concept of optimality and highlights the differences.
In~\cite{BjorkKhapkoMurgoci.17,BjorkMurgoci.14} the authors study time-inconsistent control in discrete and continuous time, respectively, and the relation between them, for a general class of objectives that are a sum of an expected utility and a nonlinear function of an expected utility with possible dependence on the initial condition. See also~\cite{Yong.12} for a continuous-time framework with dependence on the initial condition.

The closest reference for the present work is~\cite{HuangZhou.17} where the authors study optimal stopping in discrete time under non-exponential discounting in a Markovian context. In the finite horizon case, a backward recursion yields the unique equilibrium. In the infinite horizon case, the authors focus on a time-homogeneous Markov chain. Under the assumption of decreasing impatience (including hyperbolic discounting), a time-homogeneous equilibrium is constructed by iterating the ``strategic reasoning'' or ``fictitious play'' map (cf.~$\Phi$ in Section~\ref{se:equilib}); that is, every agent optimizes her decision between continuing and stopping while taking as given the decisions of all other agents. Remarkably, an equilibrium which is optimal for all agents can be obtained. We remark that~\cite{HuangZhou.17} is predated by~\cite{HuangNguyenHuu.18} where the iterative approach was first implemented in continuous time. In~\cite{HuangNguyenHuu.18}, time-homogeneous equilibria are obtained for time-homogeneous diffusions and inhomogeneous equilibria for time-inhomogeneous diffusions. See also \cite{HuangZhou.17b} for a discussion of optimal equilibria in continuous time and \cite{TanWeiZhou.18} for a recent study of optimal stopping with non-exponential discounting where equilibria may not exist and this fact is related to a failure of smooth pasting. Optimal stopping under probability distortion is studied in \cite{HuangNguyenHuuZhou.17} with a particular focus on equilibria that are obtained by iterating from na{\"i}ve strategies.

The mentioned works on optimal stopping in continuous time use a direct analogy to the discrete-time case to define equilibria: each agent may stop or continue, without any commitment device. Indeed, for optimal stopping, the first-order approach of \cite{EkelandLazrak.06} is not a necessity: the decision to stop at a single instance in time immediately affects the process. On the other hand, as highlighted by~\cite{EbertStrack.18} in the context of prospect theory, the definition in continuous time may include unreasonable equilibria based on the fact that continuation and stopping for a time-$t$ agent produce the same payoff if the subsequent agents stop and $G$ is continuous. In particular, ``always stopping'' is an equilibrium even if, say, $G$ is increasing. 
In a homogeneous diffusion model, \cite{ChristensenLindensjo.18a,ChristensenLindensjo.18} use a first-order condition to define equilibria for two problems with time-inconsistency, and then ``always stopping'' is not necessarily an equilibrium. The relation between the two definitions has not been clarified so far.

To the best of our knowledge, the present paper is the first investigation of conditional optimal stopping. Regarding the control of conditioned processes, we would like to mention ongoing work of R.~Carmona and M.~Lauri\`ere where the problem of~\cite{Lions.17} is studied as a mean field control problem for open and closed loop controls as well as ongoing work of Y.~Achdou and M.~Lauri\`ere on the numerical resolution.

\subsection{Synopsis} 

We study the conditional optimal stopping problem in~\eqref{eq:introPrecomm} in a discrete-time setting with finite or infinite time horizon. While a continuous-time setting may certainly be of interest, our choice avoids some of the difficulties mentioned in the preceding section and leads to an uncontroversial definition of an equilibrium: at every time and state $(t,\omega)$, an agent makes a binary choice---stopping or continuing---without committing future agents.
We analyze such equilibria in a general stochastic framework while paying particular attention to the Markovian setting. 

In the case of a finite time horizon~$T$, there is a natural terminal condition (stopping is mandatory at~$T$) and we shall see that there is an equilibrium which can be constructed by a backward recursion. This recursion computes two processes, a value process like in the classical case and an additional ``survival process'' that keeps track of the conditioning probability induced by the future selves' decisions. The equilibrium is essentially unique, and if the stochastic framework is Markovian, then so is the equilibrium. These findings are in line with the results for other time-inconsistent problem as described in Section~\ref{se:lit}.

In the case of an infinite horizon, we provide a fairly general existence result by passing to the limit of finite horizon problems. (Note that for non-exponential discounting, existence may fail if the discounting does not satisfy decreasing impatience;  cf.~\cite[Example 3.1]{HuangZhou.17}.) On the other hand, we also provide examples showing that this case is more subtle than the previous one. We shall see that there can exist non-Markovian equilibria in addition to Markovian ones in a Markovian setting, which disproves a conjecture of~\cite{BjorkMurgoci.14} for our problem. Moreover, equilibria need not be unique even within the class of Markovian equilibria. Even more surprisingly, we detail a time-homogeneous Markovian example which does not admit a time-homogeneous equilibrium while time-inhomogeneous equilibria do exist. This is in sharp contrast to the results of~\cite{HuangNguyenHuu.18, HuangZhou.17} and illustrates that for our problem, in general, iterating the ``strategic reasoning'' map of~\cite{HuangZhou.17} does not converge. At a technical level, one reason is that non-exponential discounting with decreasing impatience as in~\cite{HuangZhou.17} preserves one inequality of the dynamic programming principle whereas in our problem, the rescaling due to the conditioning probability can cause deviations in both directions. 

It seems natural to ask for analogues of the classical Snell envelope theory in our setting. Indeed, the two processes described in the recursion for the finite time horizon can be characterized in more abstract terms by supermartingale properties. This leads to a notion that we call Snell pair and extends to the infinite-horizon setting. Snell pairs are (essentially) in one-to-one relation with equilibria. Similarly as in the classical case, the equilibrium policy is retrieved from the Snell pair by stopping where the value process meets the obstacle $G$, but the survival process is needed to adjust the classical supermartingale properties in the context of conditioning. The survival process, in turn, also enjoys a supermartingale property. We are not aware of similar notions in the prior literature.

The remainder of this paper is organized as follows. In Section~\ref{se:setting} we detail the observation of time-inconsistency and the equilibrium concept. Section~\ref{se:finiteHorizon} presents the results on the finite-horizon case. Existence of equilibria in the infinite-horizon case is covered in Section~\ref{se:infiniteHorizonExistence} and the corresponding examples are described in Section~\ref{se:infiniteHorizonExamples}. The concluding Section~\ref{se:Snell} discusses Snell pairs and their relation to equilibria.

\section{Setting}\label{se:setting}

Let $T\in\N\cup\{\infty\}$ be the time horizon. If $T<\infty$, set $\T=\{0,1,2,\ldots, T\}$; if $T=\infty$, set $\T=\N$. We will work on a probability space $(\Omega, \cF, P)$ equipped with a filtration $(\cF_{t})_{t\leq T}$ such that $\cF_{0}$ is trivial. Let $\sigma$ be a stopping time with $P(\sigma>0)=1$; we think of events that happen after $\sigma$ as irrelevant and call $D_{t}:=\{t<\sigma\}$ the \emph{domain of relevance} at time~$t\in\T$. In the case $T=\infty$, it is convenient to set $D_{\infty}:=\cap_{t\in\T} D_{t}=\{\sigma=\infty\}$.  We may note that $\sigma(\omega)=\inf\{t\in\T: \omega \notin D_t\}$; indeed, specifying $\sigma$ is equivalent to specifying a decreasing adapted sequence $(D_t)_{t\in\T}$ with $P(D_{0})=1$. Here and in what follows, the convention $\inf \emptyset =\infty$ is used.
Finally, let $G=(G_t)_{t\leq T}$ be an adapted process describing the payoff for stopping at time $t$. The value of $G_{t}$ outside $D_{t}$ will not matter; we set $G_t=\Delta$ on $D_t^c$ for notational purposes, where $\Delta$ is an auxiliary state with the convention that $0\cdot \Delta=0$. We assume throughout that 
$E[\sup_{t\leq T} |G_t|1_{D_{t}}]<\infty$. 
Since we are interested in events that happen strictly before $\sigma$, including the case where $\sigma$ never happens, it will be useful to introduce the notation
\[
  s \lhd t \quad\Longleftrightarrow\quad  s < t \;\mbox{ or }\; t=\infty \qquad \mbox{for}\quad s,t\in [0,\infty].
\]
We can then consider the \emph{precommitted} optimal stopping problem at the initial time,
\begin{equation}\label{eq:precommProblem}
  V_{pre}=\sup_{\tau\le T,\, P(\tau\lhd \sigma)>0} \frac{E[G_\tau 1_{\{\tau\lhd \sigma\}}]}{P(\tau\lhd \sigma)}.
\end{equation}
Note that the supremum only runs over stopping times $\tau$ which avoid conditioning on a nullset and that the set of such times always includes $\tau\equiv0$.

\begin{example}[Markovian Setting]\label{ex:MarkovChain}
  Let $X$ be a Markov chain with values in a separable metric space $\X$ starting at $X_{0}=x_{0}$, let $B\subseteq \X$ be a measurable subset containing $x_{0}$ and let $\sigma=\inf\{t\geq0:\,X_{t}\notin B\}$ be the first exit time from~$B$. Then, our model entails that we only evaluate states of the world where the trajectory of~$X$ lies in~$B$ up to the stopping time~$\tau$. A possible specification of the payoff is $G_{t}=\delta^{t}g(t,X_{t})$ for a deterministic function $g$ and a discount factor $\delta\in(0,1]$. More generally, the set $B$ can be time-dependent.
\end{example}

The conditional optimal stopping problem~\eqref{eq:precommProblem} reduces to a classical optimal stopping problem when $\sigma=\infty$. But in general, the conditioning in the definition of the expected payoff for~$\tau$ depends on~$\tau$ itself, so that it cannot be reduced to a classical stopping problem.

\subsection{Equilibria}\label{se:equilib}

The following example illustrates that the optimization problem~\eqref{eq:precommProblem} is time-inconsistent in the sense that an optimal stopping strategy for an agent today may not be optimal in the future; that is, if she reconsiders her strategy at a future time using a conditional criterion, she may contradict her previous decision.

\begin{example}\label{ex:timeInconsistent}
  Consider a two-period binomial tree with $\Omega=\{uu, ud, du, dd\}$ as illustrated in Figure~\ref{fig:eg1}, where $u$ stands for up and $d$ for down. The conditional probabilities are $1/2$ on every edge and the numbers at each node represent the payoff $G$. The domain of relevance includes all states except~$dd$; i.e., the dashed line indicates the exit from the domain.
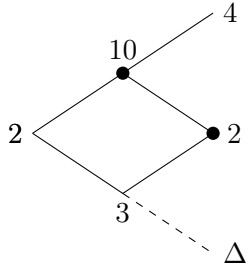
\begin{figure}[h]
\begin{center}
\begin{tikzpicture}[scale=0.8, dot/.style={circle,fill=black,minimum size=5pt,inner sep=0pt,outer sep=-1pt}]
\draw (0,0)node[left]{2} -- (1.5,1)node[dot,label=above: 10]{}--(3,2)node[right]{4};
\draw (0,0)node[left]{2} -- (1.5,-1)node[below]{3}--(3,0)node[dot,label=right: 2]{}--(1.5,1);
\draw[dashed] (1.5,-1)--(3,-2)node[right]{$\Delta$};
\end{tikzpicture}
\caption{The binomial tree of Example~\ref{ex:timeInconsistent}.}
\label{fig:eg1}
\end{center}
\end{figure}
Since there are only five distinct stopping times in this model, once can easily compute all possible payoffs and observe that the unique optimizer of~\eqref{eq:precommProblem} is the stopping time $\tau_{pre}$ with $\tau_{pre}(uu)=\tau_{pre}(ud)=1$ and $\tau_{pre}(du)=\tau_{pre}(dd)=2$. To wit, it is optimal to stop at $t=1$ if we have moved up in the first step and at $t=2$ otherwise. The obtained payoffs are illustrated by the solid dots and the associated value is $V_{pre}=10\cdot \frac{2}{3}+2\cdot \frac{1}{3}=\frac{22}{3}$.

Next, consider an analogous optimization problem for an agent who solves the problem conditionally on starting in the down state at $t=1$. This agent has only two options, either to stop immediately with payoff 3 or to wait until the horizon and receive an expected reward of $2$ (since the expectation is conditioned on remaining inside the domain). Thus, this agent prefers to stop, and that is not consistent with $\tau_{pre}$. In summary, if the first agent solves~\eqref{eq:precommProblem}  and reconsiders her own strategy at $t=1$ in the down state using the natural conditional criterion, she will overturn her previous decision.
\end{example}

For the remainder of the paper we focus on an uncommitted sophisticated agent in the sense of~\cite{Strotz.55} (see~\cite{KarnamMaZhang.17} for a recent paper surveying other approaches). She thinks of her ``future selves'' at various times and states as other agents that will optimize their choices when subsequent decisions are considered as given. Thus, we look for a policy which future selves will not override. A policy is a collection of binary decisions (stop or continue), one for each time and state, and an equilibrium is a policy such that no agent is incentivized to deviate.

Before formalizing this, let us observe that each agent faces the constraint of not conditioning on a null event. That is, any agent is forced to stop if continuing would lead to exiting the domain with probability one in the next step. 
Thus, the problem has the (random) \emph{effective time horizon}
\[
  T_{e}:=T\wedge \inf\{0\le t<T:\,P(D_{t+1}|\cF_{t})=0\}.
\]
The following adapts the basic notions of \cite{HuangNguyenHuu.18, HuangZhou.17} to our problem of conditional stopping (instead of non-exponential discounting) and extends them to a non-Markovian setting.

\begin{definition}
  A \emph{stopping policy} is a $\{0,1\}$-valued adapted process $\theta=(\theta_{t})_{t\in\T}$. We interpret $\theta_{t}(\omega)=1$ as the agent at $(t,\omega)$ choosing to stop and $\theta_{t}(\omega)=0$ as continuing.
We also introduce the \emph{continuation stopping time}
\[
  \cL_{t}\theta=\inf\{s>t: \theta_s=1\};
\]
this is the stopping time induced by $\theta$ for a time-$t$ agent who decides to continue. A stopping policy $\theta$ is called \emph{admissible} if 
\[
  P(\cL_{t}\theta\lhd \sigma |\cF_{t})>0\quad\mbox{for}\quad t<T_{e} \qquad\mbox{and}\qquad  \theta_{t}=1\quad\mbox{for}\quad t\geq T_{e}.
\]
We denote by $\Theta$ the set of all admissible stopping policies.
\end{definition}

Admissibility implies that every time-$t$ agent with $t<T_{e}$ has a well-defined \emph{continuation value}
\begin{align*}
  J_{t}(\theta)
  =\frac{E[G_{\cL_{t}\theta} 1_{\{\cL_{t}\theta \lhd \sigma\}}|\cF_{t}]}{P(\cL_{t}\theta \lhd \sigma|\cF_{t})}, \quad t<T_{e}.
\end{align*}
Naturally, she compares $J_{t}(\theta)$ with her stopping value $G_{t}$ and prefers the larger one, or she is invariant if they are equal. (Agents with $t=T_{e}$ are forced to stop, so there is no decision to be taken. The value of $\theta_{t}$ for $t>T_{e}$ is unimportant and set to $1$ only for specificity.) If we start with some $\theta\in \Theta$ and all agents simultaneously update their choice according to this preference while using the convention that invariant agents stick to their preexisting decision, we are led to the updated stopping policy
\[
  \Phi(\theta)_t=\begin{cases}
    1 & \mbox{if } t< T_{e} \mbox{ and } G_t>J_{t}(\theta),\\
    \theta_{t} & \mbox{if } t< T_{e} \mbox{ and } G_t=J_{t}(\theta),\\
    0 & \mbox{if } t< T_{e} \mbox{ and } G_t<J_{t}(\theta),\\
    1 & \mbox{if } t\geq T_{e}.
  \end{cases} 
\]  

\begin{definition}
  An admissible stopping policy $\theta$ is an \emph{equilibrium (stopping policy)} if $\Phi(\theta)=\theta$.
\end{definition}

This notion corresponds to a subgame perfect Nash equilibrium: each agent is behaving optimally if the future agents' choices are seen as given. 

\begin{example}\label{ex:equilibBinomEx}
  Consider the setting of Example~\ref{ex:timeInconsistent}. In any admissible stopping policy, the time-$2$ agents have to stop because of the time horizon. Both time-$1$ agents then prefer to stop as their stopping values (10 and 3) exceed the expected continuation values (3 and 2). Given those decisions, the expected continuation value for the time-$0$ agent is $(10+3)/2$ which exceeds the stopping value of $2$. It easily follows that the unique equilibrium stopping policy is given by $\theta_{0}=0$, $\theta_{1}\equiv1$ and $\theta_{2}\equiv1$. The induced stopping time for the time-$0$ agent is $\tau\equiv1$. This differs from the precommitted-optimal stopping time $\tau_{pre}$ of Example~\ref{ex:timeInconsistent}, and the associated expected reward of $(10+3)/2$ is smaller than the precommitted value function $V_{pre}$.
\end{example} 

In a Markov chain setting, a natural subset of stopping policies is also of a Markovian form. Denoting by $\sigma(Y)$ the $\sigma$-field generated by a random variable~$Y$, this can be formalized as follows.

\begin{definition}
  Consider the Markovian setting of Example~\ref{ex:MarkovChain}. A stopping policy $\theta\in\Theta$ is called \emph{Markovian} if $\theta_{t}$ is $\sigma(X_{t},1_{D_{t}})$-measurable for all $t\in\T$. 
\end{definition}

If $\theta$ is admissible, this is equivalent to the existence of measurable subsets $R_{t}\subseteq \X$ such that 
\[
  \theta_{t}=1_{\{X_{t}\in R_{t}\} \cup D_{t}^{c}}.
\]
Note that such equilibria are actually path-dependent through $D_{t}$, but this is the least amount of path-dependence compatible with our general definition of admissibility. In the Markovian setting, one could assume without loss of generality that all exit states (states outside $B$) are absorbing. Then, we have $D_{t}=\{X_{t}\in B\}$ a.s.\ and one can require that $\theta_{t}$ is (a.s.) $\sigma(X_{t})$-measurable.

\section{Finite-Horizon Equilibria}\label{se:finiteHorizon}

In this section we discuss existence, uniqueness and construction of equilibria for the case $T<\infty$. 

  In the classical optimal stopping problem, the value function and the optimal decision of a time-$t$ agent are completely determined by the value functions of the agents at time $t+1$. This fact lies at the heart of the backward recursion of dynamic programming and the Snell envelope theory. In the problem at hand, however, the conditioning event in the computation of the continuation value $J_{t}(\theta)$ depends on the decisions of many future selves, not only the ones at time $t+1$. This suggests introducing an additional process~$S$ to keep track of the probability of the conditioning event given the stopping policy of all future selves; we call~$S$ the \emph{survival process} since it is related to survival probabilities. In Theorem~\ref{th:backwardAlgo} below we provide a backward recursion to construct an equilibrium; its recursive formula for $J_{t}(\theta)$ resembles the classical case where it would be the conditional expectation of the value process at time $t+1$, but now this expectation is calculated under a new measure obtained by using the normalized survival process as a density.
  
Just like in classical optimal stopping, one type of non-uniqueness arises when an agent is invariant; that is, when the stopping and continuation values happen to be equal: $J_{t}(\theta)=G_{t}$. Thus, an algorithm for the construction of an equilibrium necessarily comes with a specific choice. The theorem stated below uses \emph{early stopping preference,} meaning that invariant agents choose to stop, and it yields the unique equilibrium with that preference. In the classical setting, this corresponds to the first time that the Snell envelope hits the obstacle. In general, a \emph{stopping preference} is an adapted process with binary values, defining for each $(t,\omega)$ the choice in the case of invariance. For each such preference, one can write an algorithm similar to Theorem~\ref{th:backwardAlgo} and it delivers the unique equilibrium with that preference. Conversely, every finite-horizon equilibrium arises in that way.

\begin{theorem}\label{th:backwardAlgo}
  Let $T<\infty$ and recall that $G_t=\Delta$ on $D_t^c$. Define the value process $(V_{t})_{t\leq T}$ and the survival process $(S_{t})_{t\leq T}$ as follows. Set $V_T=G_T$ and $S_T=1_{D_T}$. For $t=T-1,\dots,0$, set   
  \[
    J_{t}=\frac{E[S_{t+1}V_{t+1}|\cF_{t}]}{E[S_{t+1}|\cF_{t}]}\quad \mbox{if} \quad t<T_{e},
  \]
  \[
  \begin{cases}
    V_{t}=G_{t} \mbox{ and } S_{t}=1 & \mbox{if } t< T_{e} \mbox{ and } G_t\geq J_{t},\\
    V_{t}=J_{t} \mbox{ and } S_{t}=E[S_{t+1}|\cF_{t}] & \mbox{if } t< T_{e} \mbox{ and } G_t< J_{t},\\
    V_{t}=G_{t} \mbox{ and } S_{t}=1_{D_{t}} & \mbox{if } t\geq T_{e}.
  \end{cases}
  \]
Then $\theta:=1_{\{G_{t}\geq V_{t}\}}$ is the unique equilibrium with preference for early stopping.
\end{theorem} 

In Section~\ref{se:Snell} we will call $(V,S)$ a \emph{Snell pair} and discuss its connection to Snell envelopes. A generalization including the infinite-horizon case will also be provided. We nevertheless opt to provide an elementary and self-contained treatment of the finite-horizon in the present section.

\begin{proof}[Proof of Theorem~\ref{th:backwardAlgo}.]
  We show in Lemma~\ref{le:survivalExpr} below that $\theta$ is admissible and that $J_{t}$ coincides with the continuation value $J_{t}(\theta)$ of~$\theta$. Once that is established, the very definition of $\theta$ shows that
  $$
  \theta=1_{\{G_{t}\geq V_{t}\}}= 
    \begin{cases}
    0 & \mbox{if } t< T_{e} \mbox{ and } G_t< J_{t}(\theta),\\
    1 & \mbox{otherwise}
  \end{cases}
$$
and hence $\theta$ is an equilibrium stopping policy with early stopping preference. On the other hand, the boundary condition at~$T_{e}$ and a backward induction allow us to see that there is at most one such equilibrium.
\end{proof}

\begin{lemma}\label{le:survivalExpr}
  In the setting of Theorem~\ref{th:backwardAlgo}, $\theta$ is admissible and
  \begin{align}\label{eq:survivalCondExp}
    J_{t}& = J_{t}(\theta), \quad t<T_{e}, \nonumber \\
    E[S_{t+1}|\cF_{t}] &= P(\cL_{t}\theta \lhd \sigma|\cF_{t}), \quad t<T_{e},
  \end{align}
  and for $t\leq T$ we have
  \[
    S_{t}=\begin{cases}
    P(\cL_{t}\theta \lhd \sigma|\cF_{t}) & \mbox{on } D_{t}\cap\{\theta_{t}=0\},\\
    1 & \mbox{on } D_{t}\cap\{\theta_{t}=1\},\\
    0 & \mbox{on } D_{t}^{c}.
  \end{cases}  
  \]
\end{lemma}

\begin{proof}
  We first check that $\theta$ is admissible. Indeed, we have $\theta_{t}=1$ for $t\geq T_{e}$, and if $t<T_{e}$, backward induction shows that $P(\cL_{t}\theta\lhd \sigma |\cF_{t})>0$.
  
  Next, we prove the formula for $S_{t}$. The last two cases are clear from the definition. Thus, we focus on showing $S_{t}=P(\cL_{t}\theta \lhd \sigma|\cF_{t})$ on $D_{t}\cap\{\theta_{t}=0\}$. For $t\geq T_{e}$ we have $\theta_{t}=1$ so nothing needs to be proved. For $t<T_{e}$ we argue by induction. Indeed, using the induction hypothesis to obtain~$(a)$ below,
  \begin{align*}
  &S_{t}
  = E[S_{t+1}|\cF_{t}]\\
  &\stackrel{(a)}{=} E\big[1_{D_{t+1}}1_{\{\theta_{t+1}=0\}} P(\cL_{t+1}\theta \lhd \sigma|\cF_{t+1}) + 1_{D_{t+1}}1_{\{\theta_{t+1}=1\}} \cdot 1 + 1_{D_{t+1}^{c}} \!\cdot 0\big|\cF_{t}\big] \\
  &\stackrel{(b)}{=} E[P(\cL_{t}\theta \lhd \sigma|\cF_{t+1})|\cF_{t}]
  =P(\cL_{t}\theta \lhd \sigma|\cF_{t}),    
  \end{align*} 
  where~$(b)$ holds due to
  \begin{equation}\label{eq:proofSurvivalRec}
  P(\cL_{t}\theta \lhd \sigma|\cF_{t+1})
  =\begin{cases}
    P(\cL_{t+1}\theta \lhd \sigma|\cF_{t+1}) & \mbox{on } D_{t+1}\cap \{\theta_{t+1}=0\},\\
    1 & \mbox{on } D_{t+1}\cap \{\theta_{t+1}=1\},\\
    0 & \mbox{on } D_{t+1}^{c}.
  \end{cases}
\end{equation}
   In the last identity, the first case holds since $\theta_{t+1}=0$ implies that $\cL_{t}\theta$ and $\cL_{t+1}\theta$ agree. The second case holds because  $\theta_{t+1}=1$ entails that $\cL_{t}\theta=t+1$ and $t+1<\sigma$ on $D_{t+1}$. Finally, on $D_{t+1}^{c}$ we have $\sigma\leq t+1\leq \cL_{t}\theta$. This completes the proof for $S_{t}$ and we note that~\eqref{eq:survivalCondExp} was obtained as part of the first display above.
  It remains to show that 
  \[
    J_{t}(\theta)\equiv\frac{E[G_{\cL_{t}\theta} 1_{\{\cL_{t}\theta \lhd \sigma\}}|\cF_{t}]}{P(\cL_{t}\theta \lhd \sigma|\cF_{t})} = \frac{E[S_{t+1}V_{t+1}|\cF_{t}]}{E[S_{t+1}|\cF_{t}]}\equiv J_{t}, \quad t<T_{e}.
  \]
  Since the denominators are non-zero and agree by~\eqref{eq:survivalCondExp}, it suffices to show
  \begin{equation}\label{eq:proofContRecClaim}
    E[G_{\cL_{t}\theta} 1_{\{\cL_{t}\theta \lhd \sigma\}}|\cF_{t}]=E[S_{t+1}V_{t+1}|\cF_{t}],\quad t < T.
  \end{equation}
  Indeed, \eqref{eq:proofContRecClaim} is clear for $t\geq T_{e}$ since that implies $P(\cL_{t}\theta \lhd \sigma)=0$. It is also clear for $t=T-1$. For $t< T_{e}\wedge (T-1)$ we argue by backward induction. We first observe that, by similar arguments as below~\eqref{eq:proofSurvivalRec},
  \begin{equation}\label{eq:proofContRec}
  G_{\cL_{t}\theta} 1_{\{\cL_{t}\theta \lhd \sigma\}}
  =\begin{cases}
    G_{\cL_{t+1}\theta} 1_{\{\cL_{t+1}\theta \lhd \sigma\}} & \mbox{on } D_{t+1}\cap \{\theta_{t+1}=0\},\\
    V_{t+1}=S_{t+1}V_{t+1} & \mbox{on } D_{t+1}\cap \{\theta_{t+1}=1\},\\
    0=S_{t+1}V_{t+1} & \mbox{on } D_{t+1}^{c}.
  \end{cases}
  \end{equation}  
  On the set $D_{t+1}\cap \{\theta_{t+1}=0\}$ occurring in the first case of~\eqref{eq:proofContRec} we have
  \[
    E[G_{\cL_{t+1}\theta} 1_{\{\cL_{t+1}\theta \lhd \sigma\}}|\cF_{t+1}] = E[S_{t+2}V_{t+2}|\cF_{t+1}] 
    = S_{t+1}J_{t+1}= S_{t+1}V_{t+1},
  \]
  where the three equalities follow from the induction hypothesis, the definitions of $J_{t+1}$ and $S_{t+1}$, and $J_{t+1} = V_{t+1}$ on $\{\theta_{t+1}=0\}$, respectively. 
  As a result, we can take conditional expectations in~\eqref{eq:proofContRec} and obtain that the identity $E[G_{\cL_{t}\theta} 1_{\{\cL_{t}\theta \lhd \sigma\}}|\cF_{t+1}] = S_{t+1}V_{t+1}$ holds everywhere. The tower property then yields the claim~\eqref{eq:proofContRecClaim} and the proof is complete.
\end{proof} 

\begin{corollary}\label{co:finiteHorizonMarkov}
  In the Markovian setting of Example~\ref{ex:MarkovChain} with $T<\infty$, there exists a unique equilibrium with preference for early stopping and that equilibrium is Markovian.
\end{corollary}

\begin{proof}
  We observe that $G_{t}$ and $V_{t}$ in Theorem~\ref{th:backwardAlgo} are $\sigma(X_{t},1_{D_{t}})$-measurable for all $t$, and then so is $\theta_{t}$.
\end{proof} 

One can note that the stopping preference is important in the above result: it is easy to construct examples of non-Markovian equilibria by specifying a path-dependent stopping preference and taking the reward function~$g$ to be constant.

\section{Infinite-Horizon Equilibria: Existence}\label{se:infiniteHorizonExistence}

The following result establishes the existence of infinite-horizon equilibria in a setting that includes Markov chains with a countable state space.

\begin{theorem}\label{th:infiniteHorizonExistence}
  Suppose that $\cF_{t}$ is a.s.\ discrete\footnote{We call a $\sigma$-field discrete if it is generated by a countable partition of $\Omega$. In the case of a Markov chain with countable state space one can define $\cF_{t}$ as the $\sigma$-field generated by the sample paths up to time $t$.} for all $t\in\T$ and that $\lim_{t\to\infty}G_{t}=G_{\infty}$ a.s. Moreover, assume that  
  \begin{equation}\label{eq:expDecay}
    \begin{array}{c}
         \mbox{$P(\exists\, t\in\T:\, G_{t}\geq0)>0$ and there exists $c>1$ such that } \\[.2em]
         \mbox{$(c^{t}G_{t})_{t\geq0}$ is uniformly bounded from above.}
    \end{array}  
  \end{equation}
  Then an equilibrium exists.
\end{theorem}

Let us comment on the assumptions before stating the proof.

\begin{remark}\label{rk:infiniteHorizonExistence}
  (a) Condition~\eqref{eq:expDecay} covers in particular problems with discounting for a payoff function with sub-exponential growth. Consider for instance the Markov chain setting of Example~\ref{ex:MarkovChain} with a bounded and nonnegative payoff function $g(t,x)$ and a discount factor $\delta\in(0,1)$. Then setting $G_{t}=\delta^{t}g(t,X_{t})$ for $t\in\T$ (and $G_{\infty}=0$), we see that~\eqref{eq:expDecay} is satisfied for any $c\in (1,\delta^{-1})$.

  (b) The proof of Theorem~\ref{th:infiniteHorizonExistence} below has three steps. The construction of a limiting stopping policy~$\theta$ and the verification of its optimality condition do not require~\eqref{eq:expDecay} at all. The latter is used to ensure that $\theta$ is \emph{admissible}. There are many other situations where admissibility holds, including without discounting, that can be established on a case-by-case basis, for instance the case of a Markov chain with a finite state space and a homogeneous reward $G_{t}=g(X_{t})$. Condition~\eqref{eq:expDecay} is merely one way to write a simple and fairly general result. 
Of course, $\sigma=\infty$ a.s.\ is always a sufficient condition for $P(\tau \lhd \sigma)\neq 0$, for any stopping time $\tau$.
  
  (c) Similarly, there are many cases where one can see directly from additional structure of $G$ that $\cL_{t}\theta<\infty$ a.s.\ for all $t\in\T$. In that case, $G_{\infty}$ is irrelevant. 
  
  (d) On the other hand, existence is not guaranteed without some assumption. For instance, if $T_{e}=\infty$ inside the domain but $P(\sigma<\infty)=1$ (cf.\ Example~\ref{ex:twoState} below with $p_{21}>0$), a strictly increasing reward $G$ leads to non-existence since stopping is undesirable for any agent but $\theta\equiv0$ is not admissible.
\end{remark}

\begin{proof}[Proof of Theorem~\ref{th:infiniteHorizonExistence}.]
  For $t<\infty$, let $\cA_{t}$ be the (countable) collection of atoms generating $\cF_{t}$.  Given $n\geq 1$, consider a modified problem with time horizon $n$ and let $(\theta^{n}_{t})_{0\leq t\leq n}$ be the equilibrium stopping policy obtained by applying Theorem~\ref{th:backwardAlgo} with the payoff $(G_{t})_{t\leq n}$. We also set $\theta^{n}_{t}\equiv1$ for $t\geq n$. Note that each $\theta^{n}_{t}$ is a binary sequence  $(\theta^{n}_{t}(A))_{A\in\cA_{t}}$. By a diagonal procedure we can thus find a subsequence (again denoted $\theta^{n}$) which converges to a stopping policy $\theta$ in the following sense: given $t<\infty$ and $A\in \cA_{t}$, we have $\theta^{n}_{t}(A)=\theta_{t}(A)$ for all sufficiently large $n$. If $T_{e}$ and $T_{e}^{n}$ denote the effective horizons, then $T_{e}\wedge n= T_{e}^{n}$ and thus the admissibility of $\theta^{n}$ for $n\geq 1$ implies that $\theta_{t}=1$ for $t\geq T_{e}$.
  
  To complete the proof that $\theta$ is admissible and an equilibrium, we fix  arbitrary $t_{0}\in\T$ and $A_{0}\in\cF_{t_{0}}$ and check the admissibility and optimality conditions at that state. For simplicity of notation, we assume that $t_{0}=0$ and $A_{0}=\Omega$ (the general case differs only by writing conditional expectations and probabilities). To further simplify the notation, we set $\tau=\cL_{0}\theta$ and $\tau^{n}=\cL_{0}\theta^{n}$. The convergence of $\theta^{n}$ to $\theta$ implies that $\tau^{n}\to \tau$ a.s. More precisely, this convergence is stationary on $\{\tau<\infty\}$,  yielding that 
  $1_{\{\tau^{n} < \sigma <\infty\}}\to 1_{\{\tau < \sigma <\infty\}}$ a.s. Moreover, $\{\tau \lhd \sigma\}=\{\tau < \sigma <\infty\} \cup \{\sigma=\infty\}$, where the union is disjoint, and similarly for $\tau^{n}$. It follows that
  \begin{equation}\label{eq:indicatorConv}
    1_{\{\tau^{n} \lhd \sigma\}}\to 1_{\{\tau \lhd \sigma\}}\quad\as
  \end{equation}

  \emph{Admissibility.} We must ensure that $P(\tau \lhd \sigma)\neq 0$. In view of~\eqref{eq:indicatorConv} it suffices to exhibit a reachable state where stopping happens for all large~$n$, as that will imply that $P(\tau \lhd \sigma)=\lim_{n} P(\tau^{n} \lhd \sigma)>0$.
  Indeed, by~\eqref{eq:expDecay} we can find $t\geq0$ and $A\in\cA_{t}$ with $A\subseteq D_{t}$ such that $G_{t}(A)\geq0$ and
  $$
    c^t G_t(A) \ge  \frac{1}{c} \sup_{s\ge 0, A'\in \mathcal{A}_s, A'\subseteq D_s} c^s G_s(A') \ge \sup_{s\ge t+1, A'\in \mathcal{A}_s, A'\subseteq D_s } c^{s-1} G_s(A')
  $$
  and hence
  $$
    G_{t}(A) \geq  G_{s}(A') \quad \mbox{for all} \quad s>t,\quad A'\in\cA_{s} \mbox{ with }A'\subseteq D_{s}.
  $$
  This shows that for the agent at $(t,A)$, stopping is optimal no matter what future selves do. In particular, $\theta^{n}_{t}(A)=1$ for all $n\geq t$ and thus $\tau \leq t<\sigma$ on $A$. As a result, $P(\tau \lhd \sigma)\geq P(A)>0$.
  
 \emph{Optimality.} It suffices to show that the continuation values converge at the fixed initial state; i.e.,   $J_{0}^{n}:=J_{0}(\theta^{n})\to J_{0}:=J_{0}(\theta)$. Once that is established, if $\theta_{0}=0$, then $\theta_{0}^{n}=0$ for $n$ large and hence $G_{0}\leq J_{0}^{n}\to J_{0}$ shows that $\theta_{0}=0$ is optimal, and similarly for $\theta_{0}=1$.  
  To see that
  \[
    J_{0}^{n}=
    \frac{E[G_{\tau^{n}} 1_{\{\tau^{n} \lhd \sigma\}}]}{P(\tau^{n} \lhd \sigma)}
    \to
    \frac{E[G_{\tau} 1_{\{\tau \lhd \sigma\}}]}{P(\tau \lhd \sigma)}
    =J_{0},
  \]
  note that the denominators are non-zero by admissibility and $P(\tau^{n} \lhd \sigma)\to P(\tau \lhd \sigma)$ by~\eqref{eq:indicatorConv}.
  In view of $\tau^{n}\to\tau$ a.s.\ we have $G_{\tau^{n}}\to G_{\tau}$ a.s.\ on $\{\tau<\infty\}$. As we have assumed that $G_{n}\to G_{\infty}$ a.s., this convergence holds everywhere. Using also the standing assumption that $E[\sup_{t\leq T} |G_t|1_{D_{t}}]<\infty$ and~\eqref{eq:indicatorConv}, the convergence of the numerators follows by dominated convergence.
\end{proof}

\begin{corollary}\label{co:infiniteHorizonMarkovianExistence}
  Consider the Markovian setting (Example~\ref{ex:MarkovChain}) under the conditions of Theorem~\ref{th:infiniteHorizonExistence}. Then there exists a Markovian equilibrium.
\end{corollary}

\begin{proof}
  We revisit the proof of Theorem~\ref{th:infiniteHorizonExistence}. Each of the finite-horizon problems is Markovian, so Corollary~\ref{co:finiteHorizonMarkov} shows that $\theta^{n}$ is Markovian. Since $\theta_{t}$ was constructed as a pointwise limit of $\theta^{n}_{t}$, it is again $\sigma(X_{t},1_{D_{t}})$-measurable.
\end{proof} 

We shall see in Example~\ref{ex:MinnieDonald} that this corollary cannot be improved in a time-homogeneous setting: the equilibria may nevertheless be time-dependent.

\section{Infinite-Horizon Equilibria: Examples}\label{se:infiniteHorizonExamples}

\subsection{Non-Uniqueness and Non-Markovian Equilibria}

The following example shows that in the infinite-horizon case, multiple equilibria may exist. In these equilibria, all agents' choices are uniquely determined; i.e., the non-uniqueness is not merely due to different choices of agents that are invariant between stopping and continuing. Moreover, the multiplicity arises even within the class of time-homogeneous Markov equilibria. The example also shows that non-Markovian equilibria may exist in a Markovian setting.

\begin{example}\label{ex:twoState}
   Consider a homogeneous Markov chain $X$ on the states $\{0,1,2\}$ with initial value $X_{0}=1$ and transition probabilities $(p_{ij})$ in its natural filtration. Only the states in $B=\{1,2\}$ are relevant for the agents, meaning that $\sigma=\inf\{t\geq0:\, X_{t}=0\}$ and $D_{t}=\{X_{1},\dots,X_{t}\in B\}$. The payoff $G_{t}=\delta^{t}g(X_{t})$ is given by a function $g$ of the current state and a discount factor $\delta\in(0,1)$. Specifically,
  \[
     p_{10}=p_{11}=p_{12}=1/3\quad \mbox{and}\quad g(0)=\Delta,\quad g(1)=1,\quad g(2)=a,
  \]
  where $a$ is a constant satisfying
  \[
    1<\frac{3-\delta}{2\delta}<a<\frac{2-\delta}{\delta}.
  \]
  We also assume that $p_{20}\neq1$; the other transition probabilities are arbitrary. Then, there are exactly two Markovian equilibria:
  \begin{enumerate}
    \item stop everywhere; i.e., $\theta\equiv 1$;
    \item stop if the chain is at State~2 or has exited; i.e., $\theta_{t}=1_{\{X_{t}=2\}\cup D_{t}^{c}}$.
  \end{enumerate} 
  If $p_{21}>0$, there are further, non-Markovian equilibria. %
  In these equilibria, the induced stopping time for a given agent at some state~$(t,\omega)$ coincides with the stopping time induced by~(i) or~(ii), conditionally on~$\cF_{t}$.
  \end{example} 
  
\begin{proof}
  We first note that as $a>g(1)$ and $\delta<1$, the only optimal choice for a time-$t$ agent on $\{X_{t}=2\}$ is to stop, no matter what future agents choose.
  
 (a) To see that $\theta\equiv1$ is an equilibrium, consider an agent at State~1, without loss of generality at $t=0$. Then 
  \begin{equation}\label{eq:twoStateCalc1}
    J_{0}(\theta)=\frac{\delta(p_{11}+a p_{12})}{1-p_{10}}=\frac{\delta(1/3+a/3)}{2/3}=\delta\frac{1+a}{2}<1=G_{0},
  \end{equation} 
  showing that stopping is indeed optimal and $\theta$ is an equilibrium.
  
  (b) The policy $\theta$ defined by $\theta_{t}=1_{\{X_{t}=2\}\cup D_{t}^{c}}$ is admissible. To see that it defines an equilibrium, consider again the time-$0$ agent at State~1. Let $\tau_{j}$ be the first hitting time of state $j$, so that $\sigma=\tau_{0}$ and $\tau:=\cL_{0}\theta=\tau_{0}\wedge\tau_{2}$. We have
$\{\tau\lhd\sigma\}=\{\tau_{2}<\tau_{0}\}$ a.s.\ since $P(\tau_{0}\wedge\tau_{2}=\infty)=0$. As $p_{10}=p_{12}$, the  symmetry between $\{\tau_{2}<\tau_{0}\}$ and $\{\tau_{0}<\tau_{2}\}$ yields that $P(\tau_{2}<\tau_{0})=P(\tau_{0}<\tau_{2})=1/2$ and thus $P(\tau\lhd\sigma)=1/2$. Moreover,
  \[
    E[\delta^{\tau}g(X_{\tau})1_{\tau\lhd\sigma}]
    =a\sum_{k\geq1} \delta^{k}P(\tau_{2}=k,\,k<\tau_{0})
    =a\sum_{k\geq1} \delta^{k}(1/3)^{k}=\frac{a\delta}{3-\delta}
  \]
  since $P(\tau_{2}=k,\,k<\tau_{0})=P(X_{1}=\dots=X_{k-1}=1,\,X_{k}=2)=p_{11}^{k-1}p_{12}$. It follows that
  \begin{equation}\label{eq:twoStateCalc2}
    J_{0}(\theta)=\frac{E[\delta^{\tau}G_{\tau}1_{\tau\lhd\sigma}]}{P(\tau\lhd\sigma)}=\frac{2a\delta}{3-\delta}>1=G_{0},
  \end{equation}
  showing that continuation is optimal. Thus $\theta$ is an equilibrium.
  
  (c) Let $\theta$ be a Markovian equilibrium; we show that $\theta$ must be one of the two above policies. We have already observed that any agent at State~2 must stop. The same holds for any agent at State~0, by admissibility. That is, $1_{\{X_{t}=2\}\cup D_{t}^{c}} \leq \theta_{t} \leq 1$ for all $t\in\T$.  If no other agent stops, $\theta$ is the policy of~(ii). Otherwise there exists a time-$t$ agent stopping at State~1: $\theta_{t}=1$ on $\{X_{t}= 1\}$. But then the same calculation as in~\eqref{eq:twoStateCalc1} shows that any agent at time $(t-1)$ and State~1 must also stop, etc., so that $\theta_{s}\equiv1$ for all $s\leq t$. As a result, the set of all agents at State~1 that stop can be thought of as a half-line starting at $t=0$. If this half-line is infinite, $\theta$ is the equilibrium from~(i). If not, there is some maximal $t<\infty$ where the time-$t$ agent stops, meaning that $\theta_{s}\equiv1$ for $s\leq t$ and $\theta_{s}=1_{\{X_{s}=2\}\cup D_{s}^{c}\}}$ for $s>t$. But now the calculation in~\eqref{eq:twoStateCalc2} shows that stopping is not optimal for any time-$t$ agent on $\{X_{t}=1\}$, a contradiction.
  
  (d) Next, we give an example of a non-Markovian equilibrium. Indeed, set $\theta_{0}=\theta_{1}\equiv 1$. For $t\geq 2$, we define
  \[
    \theta_{t}(\omega)=
    \begin{cases}
      0 & \mbox{if } \omega\in\{X_{1}=2,\,X_{t}=1\}\cap D_{t},\\
      1 & \mbox{else.}
    \end{cases} 
    \]
    Simple calculations analogous to~\eqref{eq:twoStateCalc1} and~\eqref{eq:twoStateCalc2} show that $\theta$ is an equilibrium. If $p_{21}>0$, both cases in the definition of $\theta$ happen with positive probability so that $\theta$ is indeed non-Markovian.
    
  (e) Let $\theta$ be any equilibrium, possibly non-Markovian. The first argument from (c) still shows that for $(t,\omega)$ such that $X_{t}(\omega)=1$ and $\theta_{t}(\omega)=1$, it follows that $\theta_{t-1}(\omega)=1$. However the second argument from~(c) merely shows that for $(t,\omega)$ such that $X_{t-1}(\omega)=X_{t}(\omega)=1$ and $\theta_{t-1}(\omega)=0$, it follows that $\theta_{t}=0$. (But this need not hold if $X_{t-1}(\omega)=2$, in contrast to the Markovian case where the policy cannot depend directly on $X_{t-1}$). This implies that given the past up to time $t$, the stopping time induced by $\theta$ is either immediate stopping as in~(i) or the first exit time of $\{1\}$ as in~(ii). Note that, as in~(d), the choice between these two may depend on $\omega$.
\end{proof} 
   
\begin{remark}\label{rk:twoStateConvergence}
  (a) The finite-horizon version of Example~\ref{ex:twoState} has a unique equilibrium, given by stopping everywhere. This follows by a backward recursion and the same calculation as in~\eqref{eq:twoStateCalc1}, since the time-$T$ agents have to stop. The limit of this equilibrium as $T\to\infty$ is the infinite-horizon equilibrium~(i). On the other hand, the equilibrium~(ii) does not arise as a limit of finite-horizon equilibria. 
  
  (b) In this particular example the two Markovian equilibria are ordered: equilibrium~(ii) has a larger value function for all agents. It is worth noting that the limit equilibrium is the inferior one.
  
  (c) Example~\ref{ex:MinnieDonald} shows that in general, no dominating equilibrium exists. One can also construct simple examples where the equilibrium value processes and stopping policies corresponding to different preferences are not ordered.
\end{remark}

\subsection{Non-Existence of Time-Homogeneous Equilibria}

In this section we construct an example of a time-homogeneous Markov chain which admits Markovian equilibria but no time-homogeneous equilibria. In that sense, Theorem~\ref{th:infiniteHorizonExistence} and Corollary~\ref{co:infiniteHorizonMarkovianExistence} cannot be improved, and a restriction  to time-homogeneous notions is not possible (or will lead to non-existence). Importantly, the example also shows that the remarkable iterative approach of~\cite{HuangZhou.17} does not apply in our setting. Indeed, in the problem of non-exponential discounting with decreasing impatience, an iterated application of $\Phi$ (from a suitable starting point) produces a monotone sequence which converges to a time-homogeneous equilibrium. In our case however, the iteration can fail to be monotone. This can be related to a failure of both inequalities of the dynamic programming principle, whereas decreasing impatience preserves one.

\begin{example}\label{ex:MinnieDonald}
Consider the homogeneous Markov chain $X$ on $\{0,1,2,3,4\}$ with transition probabilities as labeled next to the edges in Figure~\ref{fig:MinnieDonald}. In particular, States 0, 3 and 4 are absorbing. We set $B=\{1,2,3,4\}$ so that 0 is the only exit state. The payoff process is given by $G_t=\delta^t g(X_t)$ where $\delta\in (0,1)$ is the discount factor and $g(1)=a$, $g(2)=2$, $g(3)=0$, $g(4)=b$ as labeled in the boxes in Figure~\ref{fig:MinnieDonald}. To avoid trivialities, we assume that the initial position is one of the non-absorbing states, i.e., either $X_{0}=1$ or $X_{0}=2$, and we also restrict our attention to equilibria that stop at State~3.\footnote{Since State~3 is absorbing and $g(3)=0$, all policies have zero reward for an agent at State~3 who is therefore invariant. This leads to an infinity of (uninteresting) equilibria. If early stopping preference is assumed, stopping at State~3 is a consequence rather than a condition.} A Markovian equilibrium $\theta_{t}=f(t,X_{t})$ is called time-homogeneous if $f$ does not depend on $t$.

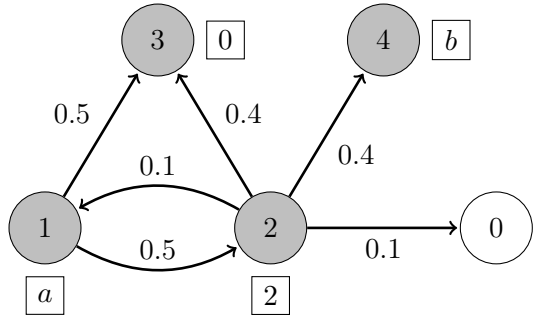
\begin{figure}[h]
\vspace{6pt}
\begin{center}
\begin{tikzpicture}
lightgray = lighten(gray, 0.5)

\node[state,fill=lightgray] at (0,0) (1){1};
\node[state,fill=lightgray] at (3,0) (2){2};
\node[state,fill=lightgray] at (1.5,2.5) (3){3};
\node[state,fill=lightgray] at (4.5,2.5) (4){4};
\node[state] at (6,0) (0){0};

\draw[every loop, line width=1pt]
(1) edge[bend right, auto=left] node {0.5} (2)
(2) edge[bend right, auto=right] node {0.1} (1)
(1) edge[auto=left] node {0.5} (3)
(2) edge[auto=right] node {0.4} (3)
(2) edge[auto=right] node {0.4} (4)
(2) edge[auto=right] node {0.1} (0);

\node at (0,-0.9) [rectangle, draw, minimum size=0.5cm] {$a$};
\node at (3,-0.9) [rectangle, draw, minimum size=0.5cm] {$2$};
\node at (2.4,2.5) [rectangle, draw, minimum size=0.5cm] {$0$};
\node at (5.4,2.5) [rectangle, draw, minimum size=0.5cm] {$b$};
\end{tikzpicture}
\end{center}
\vspace{-7pt}
\caption{The Markov chain of Example~\ref{ex:MinnieDonald}, with states $x$ labeled in circles and payoffs $g(x)$ in boxes.}
\label{fig:MinnieDonald}
\end{figure}

We fix $0<a<\delta<1<2<b$ such that the following inequalities are satisfied:
\begin{equation}\label{cond1}
a<\delta, \quad \delta(a+4b)<18,
\end{equation}
\begin{equation}\label{cond2}
\delta(\delta+4b)>18, \quad 0.01 \delta^3 \min(5a, b\delta^2)+0.2b\delta^3+4b\delta>17.9,
\end{equation}
\begin{equation}\label{cond3}
\delta^2(\max(\delta,0.25b\delta^2)+4b)<18.9 a.
\end{equation}
One possible choice is $\delta=0.999$, $a=0.96$, $b=4.257$. Then, up to a.s.\ equivalence:

\begin{enumerate}
\item All equilibria are Markovian.

\item
There exists no time-homogeneous equilibrium.

\item
There are exactly two equilibria  %
and they are given by shifts of one another. Indeed, let
\[
  \theta^1_t=f(t,X_t), \ \theta^2_t=f(t+1,X_t), \ \theta^3_t=f(t+2,X_t), \ \theta^4_t=f(t+3,X_t)
\]
where
\[f(t,x)=\begin{cases}
1_{R_4}(x), & t\equiv 0 \mod 4\\
1_{R_3}(x), & t\equiv 1 \mod 4\\
1_{R_2}(x), & t\equiv 2 \mod 4\\
1_{R_1}(x), & t\equiv 3 \mod 4\\
\end{cases}\]
for
\[
  R_1=\{0,1,2,3,4\}, \ R_2=\{0,2,3,4\}, \ R_3=\{0,3,4\}, \ R_4=\{0,1,3,4\}.
\]
Then $\theta^i$, $i=1,\dots,4$ are Markovian equilibria, and exactly two of them are distinct up to a.s.\ equivalence: if $X_{0}=1$, then $\theta^{1}=\theta^{4}$ and $\theta^{2}=\theta^{3}$, whereas if $X_{0}=2$, then $\theta^{1}=\theta^{2}$ and $\theta^{3}=\theta^{4}$, a.s.\,\footnote{Recall that the initial condition is deterministic in our basic setup. If $X_{0}=1$, then State~1 can only be visited at odd $t$ and State~2 only at even $t$; the reverse is true if $X_{0}=2$. This leads to the a.s.\ equivalence of two pairs of $\theta^{i}$. Whereas if we treated the initial state as not being fixed (as may be considered natural in a Markovian framework) or if we assumed that $X_{0}$ has a distribution with support including both states, then all four equilibria would be distinct.}

\end{enumerate}

That all equilibria are Markovian is related to the filtration being relatively small (a.s.) due to various states being absorbing---this fact should not be given too much weight. The proofs for the other items are rather lengthy, so let us try to summarize the key mechanics heuristically. First, the dynamics are engineered such that in any equilibrium, the decision of a time-$t$ agent depends only on the agents at $t+1$. Moreover, as highlighted in Lemma~\ref{le:MinnieDonald} below, it embeds two types of agents that cannot agree (and cannot even agree to disagree): Call Minnie$_{t}$ the agent at State~2 and time~$t$ and Donald$_{t}$ the agent at State~1 and time $t$. Minnie prefers to live in harmony and always wants to agree, whereas Donald is only happy if he contradicts Minnie. Suppose that at some time $t$, Donald$_{t}$ says ``1'' (stop). Then Minnie$_{t-1}$ also opts for~1, but the combative Donald$_{t-2}$ immediately replies with 0, thus implying time-inhomogeneity as he is contradicting Donald$_{t}$. The situation is similar if Donald$_{t}$ starts with~$0$. 

Conversely, there are exactly two equilibria because the above backward recursion also implies a unique forward recursion once the initial Donald$_{0}$ (or Minnie$_{0}$, depending on what the initial state $X_{0}$ is) fixes one of the two possible choices 0 or 1.
\end{example}

\begin{proof}[Proof of (i)--(iii).]
Let us first observe that any equilibrium stopping policy~$\theta$ (possibly non-Markovian) must stop on $\{X_{t}=0\}$, by admissibility. Furthermore, it must stop on $\{X_{t}=4\}$: State~4 is absorbing and $g(4)>0$, so that continuing is never optimal due to the discount factor $\delta<1$. Since we have also convened that $\theta$ stops on $\{X_{t}=3\}$, we may henceforth restrict our attention to equilibria satisfying $\theta_t=1$ on $\{X_t\in\{0,3,4\}\}$ for all $t\in\T$.

(i) Let $\theta$ be any equilibrium; we show that $\theta$ is Markovian (or rather, a.s.\ equivalent to a Markovian equilibrium). Indeed, suppose first that the initial condition is $X_{0}=1$, and fix $t\in\T$. We have that $\theta_{t}=1$ on $\{X_{t}\in\{0,3,4\}\}$. But since $0,3,4$ are absorbing states, $\{X_{t}\in\{0,3,4\}\}=\cup_{s\leq t}\{X_{s}\in\{0,3,4\}\}$.  Suppose that $t\in\T$ is odd. Then $\{X_{t}=1\}$ is a nullset, so that up to a.s.\ equivalence, only the value of $\theta_{t}$ on $\{X_{t}=2\}$ has not been determined yet. But due to the absorption on $\{0,3,4\}$ and the fact that exactly one of the sets  $\{X_{s}=1\}$ and $\{X_{s}=2\}$ has positive probability for every $s\leq t$, we have $\{X_{t}=2\}=\{X_{1}=2,\, X_{2}=1,\, X_{3}=2,\, \dots, X_{t}=2\}$ which implies that $\{X_{t}=2\}$ is an atom in $\cF_{t}$. In particular, $\theta_{t}$ is a.s.\ constant on $\{X_{t}=2\}$, and since $\theta_{t}=1$ a.s.\ on $\{X_{t}=2\}^{c}$, it follows that $\theta_{t}$ is of Markovian form. The situation is analogous if $t$ is even, and hence $\theta$ is Markovian. The initial condition is $X_{0}=2$ is dealt with similarly.

The proof of (ii) and (iii) necessitates the following lemma which describes the Minnie--Donald relationship sketched above.

\begin{lemma}\label{le:MinnieDonald}
  Let $0<a<\delta<1<2<b$ satisfy \eqref{cond1}--\eqref{cond3} and let $\theta$ be an admissible stopping policy
  such that $\theta_t=1$ on $\{X_t\in\{0,3,4\}\}$ for $t\in\T$. Then for all $t\geq1$,
  \begin{itemize}
  \item[{\normalfont (P1)}] %
  if $\theta_{t}=1$ on $\{X_t=1\}$, then $\Phi(\theta)_{t-1}=1$ on $\{X_{t-1}=2\}$;
  \item[{\normalfont (P2)}] %
  if $\theta_{t}=0$ on $\{X_t=1\}$, then $\Phi(\theta)_{t-1}=0$ on $\{X_{t-1}=2\}$;
  \item[{\normalfont (P3)}] %
  if $\theta_{t}=1$ on $\{X_t=2\}$, then $\Phi(\theta)_{t-1}=0$ on $\{X_{t-1}=1\}$;
  \item[{\normalfont (P4)}] %
  if $\theta_{t}=0$ on $\{X_t=2\}$, then $\Phi(\theta)_{t-1}=1$ on $\{X_{t-1}=1\}$.
  \end{itemize}
\end{lemma}

The proof of the lemma is reported after the proof of~(ii) and~(iii).

(ii) Define the 4-periodic sequence $(R_n)$ by $R_n=R_{n+4\mathbb{Z}}$ where $R_{1},\dots,R_{4}$ are as in (iii) above. Note that $R_{1},\dots,R_{4}$ exhaust all combinations of $\{0,3,4\}$ and the remaining states. Thus, a time-homogeneous equilibrium $\theta$ must (a.s.) be of the form $\theta_{t}=1_{R_{n}}(X_{t})$, $t\in\T$, for some $n$.  On the other hand, for any $t\in\T$, (P1)--(P4) imply that $\Phi(\Phi(1_{R_n}(X_{t})))=1_{R_{n+2}}(X_{t})\neq 1_{R_n}(X_{t})$, thus ruling out the existence of a time-homogeneous equilibrium. (We iterate $\Phi$ twice to ensure that the policies differ also modulo a.s.\ equivalence).

(iii) Admissibility of $\theta^i$ ($i=1,2,3,4$) holds since $D_{t}=\{X_{t}=0\}^{c}$ a.s.\ (due to~$\{0\}$ being absorbing) and since from any non-absorbing state there is a positive probability of reaching $\{3,4\}$ before reaching $\{0\}$. Moreover, $\Phi(\theta^i)=\theta^i$ follows by direct verification using (P1)--(P4). Hence, $\theta^i$ are equilibria.

To see that there are exactly two equilibria, suppose first that the initial condition is $X_{0}=1$ and let $\theta$ be a (necessarily Markovian) equilibrium. Modulo a.s.\ equivalence, $\theta$ is completely determined by its values on $\{X_{0}=1\}$, $\{X_{1}=2\}$, $\{X_{3}=1\}$, etc., since State~1 can only be visited at even times and State~2 only at odd times. Next, we use (P1)--(P4): Suppose that $\theta_{0}=1$ on $\{X_{0}=1\}$. This implies $\theta_{1}=0$ on $\{X_{1}=2\}$, which implies $\theta_{2}=0$ on $\{X_{2}=1\}$, which implies $\theta_{3}=1$ on $\{X_{3}=2\}$, etc. Therefore, we have $\theta=\theta^{1}=\theta^{4}$ a.s.

Alternately, $\theta_{0}=0$ on $\{X_{0}=1\}$. This implies $\theta_{1}=1$ on $\{X_{1}=2\}$, thus $\theta_{2}=1$ on $\{X_{2}=1\}$, thus $\theta_{3}=0$ on $\{X_{3}=2\}$, etc. In particular, we have $\theta=\theta^{2}=\theta^{3}$ a.s. 

The case of the initial condition $X_{0}=2$ is similar.\footnote{If the initial state is not considered fixed or if it is random with $P(X_{0}=1)>0$ and $P(X_{0}=2)>0$, then (P1)--(P4) imply that $\theta$ is a.s.\ equal to exactly one of the four $\theta^{i}$, uniquely determined by the values of $\theta_{0}$ on $\{X_{0}=1\}$ and $\{X_{0}=2\}$.}%

\end{proof}

\begin{proof}[Proof of Lemma~\ref{le:MinnieDonald}.]
Let $t\geq1$ and set
\[
  \tilde J_t(\theta)=\frac{E[\delta^{\cL_{t}\theta-t}g(X_{\cL_{t}\theta}) 1_{\{\cL_{t}\theta<\sigma\}}|\cF_t]}{P(\cL_{t}\theta<\sigma|\cF_t)}
\]
so that $J_{t}(\theta)= \delta^t \tilde J_t(\theta)$ is the continuation value at time $t$. Note that comparing $J_t(\theta)$ with $G_t$ is equivalent to comparing $\tilde J_t(\theta)$ with $g(X_t)$.
We first show (P1) and (P3).\\[-.8em]

(P1): Suppose that $X_{t-1}=2$. This implies $X_t\in\{0,1,3,4\}$ %
and thus the assumption of~(P1) yields that $\theta_t=1$, $\cL_{t-1}\theta=t$ and 
$\tilde J_{t-1}(\theta)=\delta(0.1a+0.4 b)/0.9$. By the second part of~\eqref{cond1}, we have $\tilde J_{t-1}(\theta)<2=g(2)$ and thus $\Phi(\theta)_{t-1}=1$ as claimed.\\[-.8em]

(P3): If $X_{t-1}=1$, then $X_t\in  \{2,3\}$ and the assumption of~(P3) imply $\theta_t=1$, $\cL_{t-1}\theta=t$ and 
$\tilde J_{t-1}(\theta)=\delta$. By the first part of~\eqref{cond1}, we have $\tilde J_{t-1}(\theta)>a=g(1)$ and thus $\Phi(\theta)_{t-1}=0$.\\[-.8em]

Next, we analyze (P2) and (P4). Denote by $h_t(\theta)$ and $p_t(\theta)$ the numerator and denominator of $\tilde J_t(\theta)$. It is clear that $h_{t}(\theta)\le \tilde J_{t}(\theta)\le b$ for all $t$, since $b$ is the maximum possible payoff. 
By iterated conditioning, we have that on the set $\{X_t=1\}\subseteq\{X_{t+1}\in\{2,3\}\}$,
\begin{align}
h_t(\theta)&=E[\delta^{\cL_{t}\theta-t}g(X_{\cL_{t}\theta})1_{\{\cL_{t}\theta<\sigma\}}|\cF_t] \notag \\
&= E[1_{\{X_{t+1}=2\}} \delta^{\cL_{t}\theta-t}g(X_{\cL_{t}\theta})1_{\{\cL_{t}\theta<\sigma\}}+1_{\{X_{t+1}=3\}} \delta g(3)|\cF_t] \notag \\
&= E[1_{\{X_{t+1}=2, \theta_{t+1}=1\}} \delta g(2)|\cF_t] \notag \\
& \quad + E[1_{\{X_{t+1}=2, \theta_{t+1}=0\}} \delta^{\cL_{t+1}\theta-(t+1)}\delta g(X_{\cL_{t+1}\theta})1_{\{\cL_{t+1}\theta<\sigma\}}|\cF_t] \notag \\
&=  \delta E[1_{\{X_{t+1}=2, \theta_{t+1}=1\}} 2 + 1_{\{X_{t+1}=2, \theta_{t+1}=0\}} h_{t+1}(\theta) |\cF_t], \label{eq:h1}
\end{align}
where we have used that $\cL_{t}\theta=t+1$ if $\theta_{t+1}=1$ and $\cL_{t}\theta=\cL_{t+1}\theta$ if $\theta_{t+1}=0$. Similarly, we deduce that on
$\{X_t=1\}$,
\begin{equation}\label{eq:p1}
p_t(\theta)= E[1_{\{X_{t+1}=2, \theta_{t+1}=1\}} + 1_{\{X_{t+1}=2, \theta_{t+1}=0\}} p_{t+1}(\theta) |\cF_t]+0.5,
\end{equation}
and on $\{X_t=2\}$,
\begin{align}
h_t(\theta)&=\delta E[1_{\{X_{t+1}=1, \theta_{t+1}=1\}} a + 1_{\{X_{t+1}=1, \theta_{t+1}=0\}} h_{t+1}(\theta) |\cF_t]+0.4b\delta, \label{eq:h2}\\
p_t(\theta)&=E[1_{\{X_{t+1}=1, \theta_{t+1}=1\}} + 1_{\{X_{t+1}=1, \theta_{t+1}=0\}} p_{t+1}(\theta) |\cF_t]+0.8. \label{eq:p2}
\end{align}
Equations \eqref{eq:h1}-\eqref{eq:p2} yield the following bounds: on $\{X_t=1\}$,
\begin{align}
h_t(\theta)&\le \delta E[1_{\{X_{t+1}=2\}}\max(2,h_{t+1}(\theta))|\cF_t], \label{leq:h1}\\
h_t(\theta)&\ge \delta E[1_{\{X_{t+1}=2\}}\min(2,h_{t+1}(\theta))|\cF_t],  \label{geq:h1} \\
p_t(\theta)&\ge 0.5+E[1_{\{X_{t+1}=2\}}p_{t+1}(\theta)|\cF_t],  \label{geq:p1}
\end{align}
and on $\{X_t=2\}$,
\begin{align}
&h_t(\theta) \le \delta E[1_{\{X_{t+1}=1\}}\max(a,h_{t+1}(\theta))|\cF_t]+0.4b\delta,  \label{leq:h2} \\
&h_t(\theta) \ge \delta E[1_{\{X_{t+1}=1\}}\min(a,h_{t+1}(\theta))|\cF_t] +0.4b\delta,  \label{geq:h2} \\
&0.8+E[1_{\{X_{t+1}=1\}} p_{t+1}(\theta) |\cF_t] \le p_t(\theta) \le 0.9.  \label{ieq:p2}
\end{align}

(P2): Suppose that $\theta_t=0$ on $\{X_t=1\}$. Throughout the proof of~(P2), we assume that we are on the set $\{X_{t-1}=2\}$; i.e., all statements are conditional on $X_{t-1}=2$. Then, $1_{\{X_t=1,\theta_t=0\}}=1_{\{X_t=1\}}$ and $1_{\{X_t=1,\theta_t=1\}}=0$. To establish that $\Phi(\theta)_{t-1}=0$, it suffices to show that $\tilde J_{t-1}(\theta)>g(2)=2$. To that end, we derive a lower bound for $h_{t-1}(\theta)$ and an upper bound for $p_{t-1}(\theta)$ (conditionally on $X_{t-1}=2$). Let 
\[\gamma_{t-1}:=P(X_t=1, X_{t+1}=2, \theta_{t+1}=1|\cF_{t-1})\]
and note that $0\le \gamma_{t-1} \le 0.05$. 
Starting from the fact that $h_{t+3}(\theta)\ge 0.4b\delta$ on $\{X_{t+3}=2\}$ by~\eqref{geq:h2}, we use \eqref{geq:h1} to see that on $\{X_{t+2}=1\}$,
\begin{align*}
h_{t+2}(\theta)& \ge \delta E[1_{\{X_{t+3}=2\}}\min(2,h_{t+3}(\theta))|\cF_{t+2}] \\
& \ge \delta E[1_{\{X_{t+3}=2\}}\min(2,0.4b\delta)|\cF_{t+2}]\\
&=0.5 \delta \min(2,0.4b\delta)=  \min(\delta,0.2b\delta^2),
\end{align*}
and then \eqref{geq:h2} and $a<\delta$ to deduce that on $\{X_{t+1}=2\}$,
\begin{align}
h_{t+1}(\theta)&\ge \delta E[1_{\{X_{t+2}=1\}}\min(a,h_{t+2}(\theta))|\cF_{t+1}] +0.4b\delta \nonumber\\
& \ge \delta E[1_{\{X_{t+2}=1\}}\min(a,\min(\delta,0.2b\delta^2))|\cF_{t+1}]+0.4b\delta \nonumber\\
&=0.1\delta \min(a,0.2b\delta^2)+0.4b\delta=0.1 A, \label{eq:h-lower}
\end{align}
where 
\[A:= \delta\min(a,0.2b\delta^2)+4b\delta. \]
By \eqref{eq:h2}, the assumption of (P2), \eqref{eq:h1}, and iterated conditioning, we have
\begin{align*}
&h_{t-1}(\theta)\\
&= \delta E[ 1_{\{X_{t}=1\}} h_{t}(\theta) |\cF_{t-1}]+0.4b\delta\\
& =\delta^2 E[ 1_{\{X_t=1, X_{t+1}=2, \theta_{t+1}=1\}} 2 + 1_{\{X_t=1, X_{t+1}=2, \theta_{t+1}=0\}} h_{t+1}(\theta) |\cF_{t-1}]+0.4b\delta.
\end{align*}
Substituting the lower bound~\eqref{eq:h-lower} for $h_{t+1}(\theta)$ into the above equation, 
\begin{equation*}
h_{t-1}(\theta)\ge  \delta^2(2\gamma_{t-1}+0.1A(0.05-\gamma_{t-1}))+ 0.4b\delta.
\end{equation*}
Similarly, using \eqref{eq:p2}, the assumption of (P2), \eqref{eq:p1}, iterated conditioning and \eqref{ieq:p2}, we obtain that
\begin{align*}
p_{t-1}(\theta)&=E[1_{\{X_{t}=1\}} p_{t}(\theta) |\cF_{t-1}]+0.8 \notag\\
&=E[1_{\{X_t=1, X_{t+1}=2, \theta_{t+1}=1\}} + 1_{\{X_t=1, X_{t+1}=2, \theta_{t+1}=0\}} p_{t+1}(\theta)|\cF_{t-1}] \notag\\
&\quad +0.5 P(X_t=1 |\cF_{t-1})+0.8 \notag\\
&\le E[1_{\{X_t=1, X_{t+1}=2, \theta_{t+1}=1\}} + 1_{\{X_t=1, X_{t+1}=2, \theta_{t+1}=0\}} 0.9 |\cF_{t-1}]+0.85 \notag \\
&= \gamma_{t-1}+0.9(0.05-\gamma_{t-1})+0.85=0.1\gamma_{t-1}+0.895.
\end{align*}
These two bounds yield that
\[\tilde J_{t-1}(\theta)=\frac{h_{t-1}(\theta)}{p_{t-1}(\theta)}\ge \frac{\delta^2(2\gamma_{t-1}+0.1A(0.05-\gamma_{t-1}))+ 0.4b\delta}{0.1\gamma_{t-1}+0.895}.\]
As a consequence, a sufficient condition for $\tilde J_{t-1}(\theta)>2$ is that
\[
  f(y):=\delta^2(2y+0.1A(0.05-y))+ 0.4b\delta-(0.2 y+1.79)>0 \quad\mbox{for all }y\in[0,0.05].
\]
Since $f$ is linear in $y$, this is equivalent to $f(0)>0$ and $f(0.05)>0$, which is precisely~\eqref{cond2}.\\[-.8em]

(P4): Suppose that $\theta_t=0$ on $\{X_t=2\}$, so that $1_{\{X_t=2,\theta_t=0\}}=1_{\{X_t=2\}}$ and $1_{\{X_t=2,\theta_t=1\}}=0$. We assume throughout the proof of (P4) that we are on the set $\{X_{t-1}=1\}$, and we shall establish that $\Phi(\theta)_{t-1}=1$ by showing the inequality $\tilde J_{t-1}(\theta)<g(1)$.

Proceeding similarly as in the proof of (P2), we start from the fact that $h_{t+3}(\theta)\le b$ and $p_{t+3}(\theta)\ge 0$ and apply \eqref{leq:h1}, \eqref{leq:h2}, \eqref{geq:p1} and \eqref{ieq:p2} repeatedly to derive the following bounds on $\{X_t=2\}$:
\[h_{t}(\theta) \le 0.1 \delta \max(\delta, 0.25 b\delta^2)+0.4b\delta, \quad\quad p_t(\theta)\ge 0.89.\]
Then, we use \eqref{eq:h1}, \eqref{eq:p1}, the assumption of (P4) and \eqref{cond3}  to deduce that
\begin{align*}
\tilde J_{t-1}(\theta)&=\frac{h_{t-1}(\theta)}{p_{t-1}(\theta)}=\frac{\delta E[1_{\{X_t=2\}} h_t(\theta)|\cF_{t-1}]}{E[1_{\{X_t=2\}} p_t(\theta) |\cF_{t-1}]+0.5}\\
&\le \frac{0.5\delta \{0.1 \delta \max(\delta, 0.25 b\delta^2)+0.4b\delta\}}{0.5\cdot 0.89+0.5}\\
&= \frac{\delta^2 \{ \max(\delta, 0.25 b\delta^2)+4b\}}{18.9}<a=g(1).
\end{align*}
The proof is complete.
\end{proof}

\begin{remark}\label{rk:noDiscountingMinnieDonald}
  The results in Example~\ref{ex:MinnieDonald} extend to the undiscounted case $\delta=1$ if we focus on equilibria that stop in the absorbing State~4 (or focus on equilibria with early stopping preference). The situation is the same as for State~3: without discounting, any agent at State~4 is invariant between stopping and continuing which leads to an infinity of equilibria.
\end{remark}

\section{Snell Pairs and Equilibria}\label{se:Snell}

In this section we provide a theory which extends both the Snell envelope of classical optimal stopping and the recursion from the finite-horizon case in Theorem~\ref{th:backwardAlgo}. As mentioned in Section~\ref{se:finiteHorizon}, the value process $V$ (which is the Snell envelope of $G$ in the classical case) needs to be complemented with the survival process $S$ to provide a sufficient statistic for an agent's optimality criterion. We introduce the Snell pair $(V,S)$ pragmatically in Definition~\ref{de:Snellpair} by stating the properties that will be used most often in the proofs. Alternately, both processes can be described through a more elegant Snell envelope property (Lemma~\ref{le:SnellEquiv}), whence the terminology. The main result of this section will be a correspondence between Snell pairs and equilibria; see Theorem~\ref{th:SnellEquilibrium} and its corollary.

We focus on equilibria with early stopping preference throughout this section. Other preferences could be accommodated but lead to (even) heavier notation. For the infinite-horizon case $T=\infty$, we assume throughout that 
\begin{equation}\label{eq:Ginfty}
  G_\infty=\limsup_{t\rightarrow \infty} G_t.
\end{equation}
We also recall that $\T=\{0,1,\dots\}$ if $T=\infty$, so that $\T\cup\{T\}$ will be used when the horizon is included in the index set.

\begin{definition}\label{de:Snellpair}
  A pair $(V,S)$ consisting of adapted processes $V=(V_{t})_{t\in\T}$ and $S=(S_{t})_{t\in\T\cup\{T\}}$ is said to be a \emph{Snell pair} (with early stopping preference) if the following hold:
  \begin{enumerate}
  \item $0< S_t\le 1$ on $D_t$ and $S_t=0$ on $D^c_t$ for all $t\in\T$, and $V_t= G_t$ for all $t\geq T_{e}$.\footnote{The property that $V_t= G_t$ for $t\geq T_{e}$ is in fact redundant with~(iii).}
  \item Given $S$, $V$ is the smallest adapted process which dominates $G$ %
  and renders $(SV)_{\cdot\wedge T_{e}}$ a supermartingale.\footnote{We follow the usual convention that supermartingale properties, Snell envelopes, etc., are understood on $\T$ unless explicitly mentioned; that is, $t=\infty$ is not included.}
  \item Given $V$, $S$ is the smallest nonnegative supermartingale on $\T\cup\{T\}$ satisfying $S_t=1$ on  $D_t\cap\{V_t=G_t\}$ for all $t\in\T$ as well as $S_\infty=1_{\{\sigma=\infty\}}$ if $T=\infty$.
  \item For all $t_0<T$, the process $(S^{t_0} V)_{\cdot\wedge T_{e}}$ is a supermartingale, where 
  \[
    S^{t_0}_t:=1_{\{t\neq t_0\}}S_t+1_{\{t=t_0\}} E[S_{t+1}|\cF_{t}].
  \]
  \end{enumerate}
\end{definition}

Some comments on the definition are in order before we connect Snell pairs with equilibrium stopping policies.

\begin{lemma}
  Properties (i)--(iii) imply the following ``martingale properties away from the obstacle,''
  \begin{itemize}
  \item[{(v)}] if $t<T_{e}$ and $V_t>G_t$, then $S_t=E[S_{t+1}|\cF_{t}]$ and $S_tV_t=E[S_{t+1}V_{t+1}|\cF_{t}]$.
  \end{itemize}
\end{lemma}

\begin{proof}
  If the first identity fails for some $t$, replacing $S_t$ by $E[S_{t+1}|\cF_{t}]$ yields a smaller supermartingale with the required properties, contradicting~(iii). If the second identity fails, replacing $V_t$ by $E[S_{t+1}V_{t+1}|\cF_{t}]/S_{t}$ yields a smaller process with the required properties, contradicting~(ii).
\end{proof} 

\begin{lemma}\label{le:SnellEquiv}
  Properties (i)--(iii) are jointly equivalent to the following:
  \begin{enumerate}
   \item[(i')] $S_t>0$ on $D_t$ for all $t\in\T$ and $V_t= G_t$ for all $t\geq T_{e}$.
   \item[(ii')] $(SV)_{\cdot\wedge T_{e}}$ is the Snell envelope of $(SG)_{\cdot\wedge T_{e}}$.
   \item[(iii')] $S$ is the Snell envelope of $1_{\{t<\infty\}\cap\{V_t=G_t\}\cap D_t}+1_{\{t=\sigma=\infty\}}$ on $\T\cup\{T\}$.
  \end{enumerate}
\end{lemma}

\begin{proof}
  Clearly (i) implies (i'). To see the reverse, suppose that $S_t>0$ on $D_t$. Then  $S'_t:=1_{D_t}$, $t\leq T$ is a nonnegative supermartingale. Thus, (iii') yields that $0\le S_t\le 1_{D_t}$ and~(i) follows. Given~(i'), the equivalence of (ii) and (ii') is immediate. For the equivalence of~(iii) and~(iii'), note that $U\wedge1$ is a supermartingale whenever $U$ is a supermartingale. %
\end{proof}

\begin{lemma}\label{le:terminalCond}
  (a) The processes $(SV)_{\cdot\wedge T_{e}}$ and $(S^{t_{0}}V)_{\cdot\wedge T_{e}}$ occurring in~(ii') and~(iv) are uniformly integrable.

  (b) Let $T=\infty$ and let $(V,S)$ be a Snell pair. Then
  \begin{equation}\label{eq:Sinfty}
    \lim_{t\to\infty}S_{t}=S_{\infty}=1_{\{\sigma=\infty\}},
  \end{equation}
  \begin{equation}\label{eq:SVinfty}
    \lim_{t\to\infty}(SV)_{t\wedge T_{e}}=1_{\{T_{e} \lhd \sigma\}} G_{T_{e}}.
  \end{equation}
\end{lemma}

\begin{proof}
  (a) Recall that $\sup_{t}|G_{t}|1_{D_{t}}\in L^{1}$ and that the Snell envelope of any process with an $L^{1}$-majorant is uniformly integrable. In view of of (ii'), it follows that $(SV)_{\cdot\wedge T_{e}}$ is uniformly integrable, and then so is $(S^{t_{0}}V)_{\cdot\wedge T_{e}}$.
  
  (b) We have from~(i) and~(iii) that $S$ is a bounded supermartingale with $S_{\infty}=1_{D_{\infty}}$. In particular, 
  $S_{t}\geq E[S_{\infty}|\cF_{t}]$. Passing to the limit, martingale convergence yields that $\liminf_{t} S_{t}\geq S_{\infty}$. Conversely, (i)~clearly implies that $\limsup_{t} S_{t}\leq1$ and that $\lim S_{t}=0$ on $\cup_{t} D_{t}^{c}=D_{\infty}^{c}$. Hence, \eqref{eq:Sinfty} is proved.
  
  Part (a), (ii') and the classical limit property of the Snell envelope yield that $\lim_{t\to\infty}(SV)_{t\wedge T_{e}}=\limsup_{t\to\infty}(SG)_{t\wedge T_{e}}$. Moreover, using~\eqref{eq:Sinfty} and~\eqref{eq:Ginfty}, 
  \begin{align*}
\limsup_{t\to\infty}(SG)_{t\wedge T_{e}}
  &=1_{\{T_{e}<\infty\}}1_{D_{T_{e}}} G_{T_{e}}+1_{\{T_{e}=\infty\}} \limsup_{t\rightarrow \infty} G_t\\
  &=1_{\{T_{e}<\sigma\}} G_{T_{e}}+1_{\{T_{e}
  =\sigma=\infty\}} G_\infty =1_{\{T_{e} \lhd \sigma\}} G_{T_{e}}
  \end{align*}
  and thus~\eqref{eq:SVinfty} follows.
\end{proof} 

We can now state the main result which relates Snell pairs to equilibria, thus extending the classical Snell envelope theory to conditional optimal stopping.

\begin{theorem}\label{th:SnellEquilibrium}
(a) Let $(V,S)$ be a Snell pair. Then, $\theta=1_{\{G\geq V\}}$ defines an equilibrium stopping policy with early stopping preference. Moreover,%
\begin{align}
  V_t&= 1_{\{t<T_{e}\}}\max(G_t, J_{t}(\theta))+1_{\{t\ge T_{e}\}}G_t, \quad t\in\T, \label{eq:theta-to-V} \\
  S_t&=1_{D_t}\left(1_{\{V_t=G_t\}}+1_{\{V_t>G_t\}} P(\cL_{t}\theta\lhd\sigma|\cF_{t})\right), \quad t\in\T\label{eq:theta-to-S}
\end{align}
and $S_{\infty}=\lim_{t\to\infty} S_{t}= 1_{\{\sigma=\infty\}}$ if $T=\infty$. \\[-.8em]

(b) If $\theta$ is an equilibrium stopping policy with early stopping preference, then there exists a unique Snell pair $(V,S)$ such that $\theta=1_{\{G\geq V\}}$. This Snell pair is given by \eqref{eq:theta-to-V}--\eqref{eq:theta-to-S}.
\end{theorem}

As mentioned above, Snell pairs reduce to the usual Snell envelope in the classical case.

\begin{corollary}\label{co:classicalSnell}
Suppose that $D_t=\Omega$ for all $t\in\T$. 

(i) Any equilibrium $\theta$ corresponds to optimal stopping in the classical sense: $E[G_{\tau_{t}}|\cF_{t}]=\esssup_{\tau\ge t} E[G_\tau|\cF_{t}]$ for $\tau_{t}=\inf\{s\ge t: \theta_s=1\}$.

(ii) Any Snell pair consists of $S\equiv1$ and the classical Snell envelope $V_t=\esssup_{\tau\ge t} E[G_\tau|\cF_{t}]$.
\end{corollary}

\begin{proof}
Note that $\sigma=\infty$.
Let $\theta$ be an equilibrium and $(V,S)$ the associated Snell pair. Then
\begin{align*}
V_t&=1_{\{\theta_t=1\}} G_t+1_{\{\theta_t=0\}}J_{t}(\theta)=1_{\{\theta_t=1\}} G_t+1_{\{\theta_t=0\}} E[G_{\cL_{t}\theta}|\cF_{t}]\\
&= E[1_{\{\theta_t=1\}} G_t+1_{\{\theta_t=0\}} G_{\cL_{t}\theta}|\cF_{t}]
= E[G_{\tau_t}|\cF_{t}].
\end{align*}
We have $S_T= 1$ by \eqref{eq:theta-to-V}--\eqref{eq:theta-to-S}. Since $S$ is a supermartingale dominated by~$1$, we must have $S_t= 1$ for all $t\in\T$. It follows that $SV=V$ is the Snell envelope of $SG=G$; that is,
$V_t=\esssup_{\tau\ge t} E[G_\tau|\cF_{t}]$.
\end{proof}

In the finite horizon-case, Snell pairs correspond to the processes constructed in Section~\ref{se:finiteHorizon}.

\begin{corollary}\label{co:SnellFiniteHorizon}
Let $T<\infty$. Then there exists a unique Snell pair $(V,S)$ and it is determined by the backward recursion of Theorem~\ref{th:backwardAlgo}.
\end{corollary}

\begin{proof}
  Let $(V',S')$ and $\theta$ be as in Theorem~\ref{th:backwardAlgo}. By Theorem~\ref{th:SnellEquilibrium}, there exists a unique Snell pair $(V,S)$ with $\theta=1_{\{G\geq V\}}$, and it is completely determined by~\eqref{eq:theta-to-V}--\eqref{eq:theta-to-S}. In view of Lemma~\ref{le:survivalExpr} and the definition in Theorem~\ref{th:backwardAlgo}, $(V',S')$ also satisfies \eqref{eq:theta-to-V}--\eqref{eq:theta-to-S}, thus $(V',S')=(V,S)$.
\end{proof}

We note that in the infinite-horizon case, the examples in Section~\ref{se:infiniteHorizonExamples} show that Snell pairs are not unique in general.
  
\begin{proof}[Proof of Theorem~\ref{th:SnellEquilibrium}.]
  We focus on the case $T=\infty$; the finite-horizon case is similar but simpler.

(a) Let $(V,S)$ be a Snell pair and $\theta=1_{\{G\geq V\}}$; we show that $\theta\in\Theta$ and $\Phi(\theta)=\theta$.
If $t\geq T_{e}$, then (i) implies $V_{t}=G_{t}$ and hence $\theta_t=1$ and $S_t=1_{D_t}$; see~(iii). Let $t< T_{e}$. Note that we are in $D_t$ and $t<\cL_{t}\theta\le  T_{e}\le \sigma$. If $\cL_{t}\theta=\infty$, we have $S_{\cL_{t}\theta}=S_\infty=1_{\{\sigma=\infty\}}=1$, whereas if $\cL_{t}\theta<\infty$, we have $S_{\cL_{t}\theta}=1_{D_{\cL_{t}\theta}}$. In summary,
\begin{equation}\label{eq:ScLtheta}
S_{\cL_{t}\theta}=1_{\{\cL_{t}\theta<\infty\}} 1_{D_{\cL_{t}\theta}}+1_{\{\cL_{t}\theta=\infty\}} =1_{\{\cL_{t}\theta\lhd \sigma \}}.
\end{equation}
As a consequence, recalling Lemma~\ref{le:terminalCond} for the case $\cL_{t}\theta=\infty$,
\begin{equation}\label{eq:SVcLtheta}
S_{\cL_{t}\theta}V_{\cL_{t}\theta}%
=1_{\{\cL_{t}\theta\lhd \sigma\}}G_{\cL_{t}\theta}.
\end{equation}
Next, we consider separately two cases.

\emph{Case $\theta_t=0$:} Using (v), the Optional Sampling Theorem (with the boundedness of $S$ and the uniform integrability from Lemma~\ref{le:terminalCond}) as well as \eqref{eq:ScLtheta} and \eqref{eq:SVcLtheta}, we see that
  \begin{equation}\label{eq:S}
  S_t=E[S_{\cL_{t}\theta}|\cF_{t}]=P(\cL_{t}\theta\lhd \sigma |\cF_{t})
  \end{equation}
  and
  \begin{equation}\label{eq:SV}
  S_tV_t=E[S_{\cL_{t}\theta} V_{\cL_{t}\theta}|\cF_{t}]=E[1_{\{\cL_{t}\theta\lhd\sigma \}}G_{\cL_{t}\theta}|\cF_{t}].
  \end{equation}
  In view of~(i), Equation \eqref{eq:S} yields in particular that $\P(\cL_{t}\theta\lhd\sigma |\cF_{t})=S_t>0$ since we are in~$D_t$, as required for the  admissibility of $\theta$. Moreover, as $\theta_t=0$, \eqref{eq:S} and~\eqref{eq:SV} together imply that
  \[
  G_t<V_t=\frac{S_{t}V_{t}}{S_{t}}=\frac{E[1_{\{\cL_{t}\theta\lhd\sigma\}}G_{\cL_{t}\theta}|\cF_{t}]}{\P(\cL_{t}\theta\lhd\sigma|\cF_{t})}=J_{t}(\theta).
  \]

\emph{Case $\theta_t=1$:} In this case, $S$ is a martingale from time $t+1$ to time $\cL_{t}\theta$ and hence, similarly to the previous case,
  \[
    S_{t+1}=E[S_{\cL_{t}\theta}|\cF_{t+1}]=E[1_{\{\cL_{t}\theta\lhd\sigma\}}|\cF_{t+1}].
  \]
  Taking conditional expectations on both sides, we deduce that
  \[
    S^t_t=E[S_{t+1}|\cF_{t}]= \P(\cL_{t}\theta\lhd\sigma|\cF_{t}).
  \]
  In view of $t<T_{e}$ and (i), we have $\P(S_{t+1}>0|\cF_{t})=\P(D_{t+1}|\cF_{t})>0$ which then implies $\P(\cL_{t}\theta\lhd\sigma|\cF_{t})= E[S_{t+1}|\cF_{t}]>0$ and finishes the proof of admissibility.
  Moreover, by the supermartingale property of $S^tV$, the Optional Sampling Theorem with the uniform integrability from Lemma~\ref{le:terminalCond}, and~\eqref{eq:SVcLtheta}, we have
  \begin{align*}
  S^t_tV_t&\ge E[S^t_{\cL_{t}\theta} V_{\cL_{t}\theta}|\cF_{t}]=E[S_{\cL_{t}\theta} V_{\cL_{t}\theta}|\cF_{t}]\\
  &=E[1_{\{\cL_{t}\theta\lhd\sigma\}}G_{\cL_{t}\theta}|\cF_{t}]=\P(\cL_{t}\theta\lhd\sigma|\cF_{t}) J_{t}(\theta)\\
  &=E[S_{t+1}|\cF_{t}] J_{t}(\theta)=S_t^t J_{t}(\theta).
  \end{align*}
  As $S^t_t>0$ and $\theta_t=0$, we conclude that $G_t=V_t\ge J_{t}(\theta)$.

Putting the two cases together and noting  $\{\theta_t=0\}\subseteq \{t< T_{e}\}\subseteq D_t$, we conclude that \eqref{eq:theta-to-V} and~\eqref{eq:theta-to-S} hold. We also recall that the condition on $S_{\infty}$ was already established in~\eqref{eq:Sinfty}.
Finally, \eqref{eq:theta-to-V} shows that 
\[
\Phi(\theta)_t=1_{\{t<T_{e}\}\cap \{G_t\geq J_{t}(\theta)\}} + 1_{\{t\geq T_{e}\}}=\theta_t
\]
and the proof of (a) is complete.\\[-.8em]

(b) Let $\theta$ be an equilibrium stopping policy with early stopping preference; we show that the pair $(V,S)$ defined by~\eqref{eq:theta-to-V} and~\eqref{eq:theta-to-S} is a Snell pair.

First, we check that~\eqref{eq:theta-to-S} implies $S_{\infty}:=\lim_{t}S_{t}=1_{D_{\infty}}$. Indeed, we have $P(\cL_{t}\theta\lhd\sigma|\cF_{t})\geq P(\sigma=\infty|\cF_{t})=P(D_{\infty}|\cF_{t})\to 1_{D_{\infty}}$. Thus, \eqref{eq:theta-to-S} implies $\lim_{t} S_{t}=1_{D_{\infty}}$ as desired.

We readily see that~(i') holds, so it suffices to show (ii'), (iii') and (iv).
Note that
\begin{equation}\label{eq:eq:thetaZeroIsIneq}
  \{\theta_t=0\}=\{\Phi(\theta)_t=0\}=\{t<T_{e}\}\cap\{J_{t}(\theta)>G_{t}\}=\{V_t>G_t\}.
\end{equation}
Let $t<T_{e}$ (which implies that we are in $D_t$), then
\begin{align*}
S^t_t&=E[S_{t+1}|\cF_{t}]\\
&=E[1_{D_{t+1}}\left(1_{\{V_{t+1}=G_{t+1}\}}+1_{\{V_{t+1}>G_{t+1}\}}\P(\cL_{t+1}\theta\lhd \sigma |\cF_{t+1})\right)|\cF_{t}]\\
&=E[1_{\{V_{t+1}=G_{t+1}\}\cap D_{t+1}}+1_{\{V_{t+1}>G_{t+1}\}\cap \{\cL_{t+1}\theta\lhd \sigma \}}|\cF_{t}]\\
&=E[1_{\{\theta_{t+1}=1\}\cap \{t+1\lhd \sigma \}}+1_{\{\theta_{t+1}=0\}\cap \{\cL_{t+1}\theta\lhd \sigma \}}|\cF_{t}]\\
&=E[1_{\{\theta_{t+1}=1\}\cap \{\cL_{t}\theta\lhd \sigma \}}+1_{\{\theta_{t+1}=0\}\cap \{\cL_{t}\theta\lhd \sigma \}}|\cF_{t}]\\
&= P(\cL_{t}\theta\lhd \sigma  |\cF_{t})\le S_t
\end{align*}
and
\begin{align*}
&E[S_{t+1}V_{t+1}|\cF_{t}]\\
&=E[1_{D_{t+1}}(1_{\{V_{t+1}=G_{t+1}\}}G_{t+1}+1_{\{V_{t+1}>G_{t+1}\}} \P(\cL_{t+1}\theta\lhd \sigma |\cF_{t+1})V_{t+1}|\cF_{t}]\\
&=E[1_{\{\theta_{t+1}=1\}\cap D_{t+1}}G_{t+1}+1_{\{\theta_{t+1}=0\}} \P(\cL_{t+1}\theta\lhd \sigma |\cF_{t+1})J_{t+1}(\theta)|\cF_{t}]\\
&=E[1_{\{\theta_{t+1}=1\}\cap\{t+1 \lhd \sigma \}}G_{t+1}+1_{\{\theta_{t+1}=0\}}E[G_{\cL_{t+1}\theta}1_{\{\cL_{t+1}\theta\lhd \sigma \}}|\cF_{t+1}]|\cF_{t}]\\
& =E[G_{\cL_{t}\theta} 1_{\{\cL_{t}\theta \lhd \sigma \}}|\cF_{t}]=\P(\cL_{t}\theta\lhd \sigma |\cF_{t})J_{t}(\theta)\le S^t_tV_t\le S_tV_t.
\end{align*}
This shows that $(S)_{\cdot\wedge T_{e}}$, $(SV)_{\cdot\wedge T_{e}}$ and $(S^tV)_{\cdot\wedge T_{e}}$ are supermartingales on $\T$. In particular, (iv) holds. In fact, $S$ is a supermartingale up to $T$: for any finite $t\ge T_{e}$, we have $V_t=G_t$ and $V_{t+1}=G_{t+1}$ and consequently
\[
  E[S_{t+1}|\cF_{t}]=E[1_{D_{t+1}}|\cF_{t}]\le 1_{D_t}=S_t.
\] 
As $S$ is bounded and $S_{\infty}=\lim S_{t}$, the supermartingale property up to $T$ follows.

Next, let $Y$ be the Snell envelope of $1_{\{t<\infty\}\cap\{V_t=G_t\}\cap D_t}+1_{\{t=\sigma=\infty\}}$. On the one hand, $S \ge Y$ since $Y$ is the smallest supermartingale dominating $1_{\{t<\infty\}\cap\{V_t=G_t\}\cap D_t}+1_{\{t=\sigma=\infty\}}$. On the other hand, let $t\in\T$ and define $\hat \tau:=t 1_{\{V_t=G_t\}}+\cL_{t}\theta 1_{\{V_t>G_t\}}$ as the stopping time induced by $\theta$ at $t$, then the stopping representation of the Snell envelope yields
\begin{align*}
Y_t&=\esssup_{\tau\ge t} E[1_{\{\tau<\infty\}\cap\{V_{\tau}=G_\tau\}\cap D_{\tau}}+1_{\{\tau=\sigma=\infty\}}|\cF_{t}] \\
&\ge E[1_{\{\hat\tau<\infty\}\cap\{V_{\hat \tau}=G_{\hat \tau}\}\cap D_{\hat \tau}}+1_{\{\hat \tau=\sigma=\infty\}}|\cF_{t}]\\
&=1_{\{V_t=G_t\}\cap D_t}+1_{\{V_t>G_t\}} E[1_{\{\cL_{t}\theta<\infty\}\cap D_{\cL_{t}\theta}}+1_{\{\cL_{t}\theta=\sigma=\infty\}}|\cF_{t}]\\
&=1_{\{V_t=G_t\}\cap D_t}+1_{\{V_t>G_t\}}   E[1_{\{\cL_{t}\theta<\sigma\}}+1_{\{\cL_{t}\theta=\sigma=\infty\}}|\cF_{t}]\\
&=1_{\{V_t=G_t\}\cap D_t}+1_{\{V_t>G_t\}} E[1_{\{\cL_{t}\theta\lhd \sigma\}}|\cF_{t}]
= S_t.
\end{align*}
Thus, we have shown $S=Y$ and (iii') is proved.

Similarly, let $Z$ be the Snell envelope of $(SG)_{\cdot\wedge T_{e}}$. We have $(SV)_{\cdot\wedge T_{e}}\ge Z$ since $Z$ is the smallest supermartingale dominating $(SG)_{\cdot\wedge T_{e}}$. Let $t\in\T$. On the set $\{V_t=G_t\}$, we trivially have $Z_t\ge (SG)_{t\wedge T_{e}}= (SV)_{t\wedge T_{e}}$ by the definition of $V$. 
Whereas on the set $\{V_t>G_t\}\subseteq\{t< T_{e}\}$, 
 \begin{align*}
Z_t&=\esssup_{\tau \geq t} E[(SG)_{\tau\wedge T_{e}}|\cF_{t}]
\ge \limsup_{N\rightarrow \infty} E[(SG)_{\cL_{t}\theta\wedge N }|\cF_{t}]\\
&=E[(SG)_{\cL_{t}\theta}|\cF_{t}]
=E[G_{\cL_{t}\theta} 1_{\{\cL_{t}\theta \lhd \sigma\}}|\cF_{t}]\\
&=S_t V_t=(SV)_{t\wedge T_{e}},
\end{align*}
where we have used the Dominated Convergence Theorem, \eqref{eq:theta-to-S}, \eqref{eq:eq:thetaZeroIsIneq} and the definitions of $V_t$, $S_t$ and $J_{t}(\theta)$. 
We conclude that $(SV)_{\cdot\wedge T_{e}}= Z$; that is, (ii') holds.

It remains to observe the uniqueness. Indeed, if $(V',S')$ is another Snell pair such that $\theta=1_{\{G\geq V'\}}$, then $(V',S')$ satisfies~\eqref{eq:theta-to-V} and~\eqref{eq:theta-to-S} by~(a). But~\eqref{eq:theta-to-V} and~\eqref{eq:theta-to-S} uniquely define the two processes, so we must have $(V',S')=(V,S)$.
\end{proof}

\newcommand{\dummy}[1]{}

\end{document}